\documentclass[11pt]{amsart}  
\usepackage{amsmath,amsfonts,amssymb,mathrsfs}
\usepackage[english]{babel}
\usepackage[utf8]{inputenc}
\usepackage[T1]{fontenc}
\usepackage{pgf,tikz,pgfplots}
\pgfplotsset{compat=1.15}
\usepackage{mathrsfs}
\usepackage{bbold}
\usepackage{enumitem}
\usetikzlibrary{arrows}
\newtheorem{theorem}{Th\'eor\`eme}[section]
\newtheorem{thm}[theorem]{Theorem}
\newtheorem{lemma}[theorem]{Lemma}

\newtheorem{coro}[theorem]{Corollary}

\newtheorem{remark}{Remark}[section]

\DeclareMathOperator\supp{supp}

\def\R{\mathbb R}

\def\Z{\mathbb Z}

\numberwithin{equation}{section}

\usepackage{hyperref}

\DeclareMathOperator{\dive}{div}

\usepackage{tikz}
\usetikzlibrary{decorations}
\usetikzlibrary{decorations.pathmorphing}
\usetikzlibrary{decorations.pathreplacing}
\usetikzlibrary{decorations.shapes}
\usetikzlibrary{decorations.text}
\usetikzlibrary{decorations.markings}
\usetikzlibrary{decorations.fractals}
\usetikzlibrary{decorations.footprints}
\tikzset{every picture/.style={execute at begin picture={
   \shorthandoff{:;!?};}
}}
\begin{document}
\title[Primitive Equations for the Ocean and Atmosphere]{The primitive equations for the ocean and atmosphere in anisotropic spaces}

\author[Valentin Lemarié]{Valentin Lemarié}

{\begin{center}
\begin{abstract}
In this work, we study the well-posedness of the primitive equations for the ocean and the atmosphere on two specific domains: a bounded domain $\Omega_1\mathrel{\mathop:}=(-1,1)^3$ with periodic boundary conditions, and the strip $\Omega_2\mathrel{\mathop:}=\mathbb{R}^2\times(-1,1)$ with periodic boundary conditions in the vertical direction. In a first time, we establish a global existence and uniqueness theorem for small initial data in a suitable anisotropic Besov space.

Then, we also justify, in a similar functional framework, the singular limit from the anisotropic Navier-Stokes equations to this system.
\end{abstract}
\end{center}}

\maketitle
\section{Introduction}

In an effort to predict meteorological phenomena, in 1922, the mathematician and meteorologist Richardson proposed and used, in \cite{Richardson}, a simplified version of the Navier-Stokes equations: the atmospheric primitive equations. These equations turn out to be a good model for studying large-scale flows where the vertical component of the motion is much smaller than the horizontal one. This is particularly the case for the troposphere, since the vertical scale (10–20 km) is much smaller than the horizontal scales (several thousand kilometers). For more information on the physical aspects of these equations, we refer to \cite{Haltiner}, \cite{Pedlosky}, \cite{Smagorinsky}, \cite{Washington}, and \cite{Zeng}.

Subsequently, they were applied to oceanographic models by Bryan \cite{Bryan}, observing that the ocean layer on Earth is very thin compared to the planet’s dimensions. For the modeling and physical motivation of the ocean primitive equations, we refer to \cite{Lewandowski}, \cite{Majda}, and \cite{Vallis}.

In this paper, we will mathematically study, on
\[
\begin{array}{lll}
\Omega_1 := \mathcal{T}^3 & &\text{the three-dimensional torus}, \\
& \text{or} & \\
\Omega_2 := \mathbb{R}^2 \times \mathcal{T}_v & &\text{with } \mathcal{T}_v \text{ the vertical torus},
\end{array}
\]
a particular case of the primitive equations for the ocean and atmosphere
\begin{eqnarray} \label{Equations primitives}
\left\{
\begin{array}{l}
\partial_t u^h + u\cdot \nabla u^h - \nu_h \Delta_h u^h + \nabla_h p = 0, \\
\partial_z p = 0, \\
\operatorname{div}_h u^h + \partial_z u^v = 0, \\
u^h \text{ is even with respect to the vertical variable } x_v, \\
u^v \text{ is odd with respect to the vertical variable } x_v,
\end{array}
\right.
\end{eqnarray}
where \( u = (u^h, u^v) \) with \( u^h \) the horizontal component and \( u^v \) the vertical one, 
\(\nabla_h := \begin{pmatrix} \partial_1 \\ \partial_2 \end{pmatrix} \) the horizontal gradient, 
\(\operatorname{div}_h V := \partial_1 V_1 + \partial_2 V_2 \) the horizontal divergence, 
\(\Delta_h := \operatorname{div}_h \nabla_h \) the horizontal Laplacian, and \( \nu_h > 0 \) the horizontal viscosity of the fluid. We would also like to point out that the model studied has no vertical viscosity.

We will denote by \( \Omega \) the spatial domain (either \( \Omega_1 \) or \( \Omega_2 \)), and by \( \Omega_h \) its horizontal component (either \( \mathcal{T}_h^2 \) or \( \mathbb{R}^2 \)).

The periodic vertical boundary conditions and the symmetry assumptions imply that on \( \{(x,y,z) \in \Omega, \ z = \pm 1\} \), we have:
\[
\partial_z u^h = 0, \quad \text{and} \quad u^v = 0.
\]

The mathematical analysis of \eqref{Equations primitives} dates back from the works of J.-L. Lions, Temam, and Wang \cite{Lions1}, \cite{Lions2}, \cite{Lions3} in the 1990s, who studied the global existence of weak solutions for these equations coupled with a temperature equation on a spherical shell, including vertical viscosity.

In 2001, Guillén-Gonzalez, Masmoudi, and Rodriguez-Bellido \cite{Guillen} proved local well-posedness in \( H^1 \) for \eqref{Equations primitives} with vertical viscosity and different boundary conditions: they worked on a bounded domain of the form
\begin{equation} \label{tilde Omega}
\tilde{\Omega} := S \times [-h,0],
\end{equation}
where \( S \) is a bounded surface in \( \mathbb{R}^2 \) and \( h: S \to \mathbb{R}_+ \) represents depth, with Neumann-type boundary conditions on
\[
\Gamma := \{(x,y,0) \mid (x,y) \in S\}
\]
and Dirichlet conditions on the rest of the boundary.
\\
Li and Titi studied various forms of these equations with temperature, salinity, or different boundary conditions (see for instance \cite{Titi1} for an overview). Notably, in \cite{Titi2}, they proved global existence of strong solutions for arbitrary \( H^1 \) initial data for a version of \eqref{Equations primitives} coupled with
\[
\partial_t T + u \cdot \nabla T - \Delta_h T = 0
\]
under boundary conditions similar to those in our study. However, since \( T \) is nonzero on part of the boundary, the solution \( T \equiv 0 \) is not admissible: hence, \eqref{Equations primitives} is not a special case.

In \cite{Titi3}, they also obtained global existence of strong solutions with arbitrary \( H^1 \) data for \eqref{Equations primitives} with vertical viscosity and boundary conditions on \( \tilde{\Omega} \) defined in \eqref{tilde Omega} (with constant \( h \)) given by:
\[
\left\{
\begin{array}{ll}
\text{on } \{z = 0 \text{ or } z = -h\} & \text{same conditions}, \\
\text{on } \{(x,y) \in \partial S, -h \le z \le 0\} & v \cdot \vec{n} = 0, \quad \partial_{\vec{n}} v \times \vec{n} = 0,
\end{array}
\right.
\]
where \( \vec{n} \) is the outward normal vector to the boundary.

Around 2020, global existence results in \( L^2 \)-type spaces (via maximal regularity methods) were obtained by Hieber et al. \cite{Hieber2}, \cite{Hieber1} and Giga et al. \cite{Giga2}, \cite{Giga1}, for \eqref{Equations primitives} with isotropic viscosity, on a bounded domain with periodic boundary conditions like those in this article. However, their solutions are only local in time and require high regularity for the data.

In \cite{Moi}, we improved on this by proving global existence for small data in isotropic Besov spaces close to the critical ones.

In the present article, we establish a global existence and uniqueness result for strong solutions with small initial data in a critical anisotropic Besov space introduced by Paicu in \cite{Paicu1}. These spaces reflect the system's anisotropy, and the result improves on \cite{Moi} in terms of regularity. Moreover, the vertical viscosity, crucial in \cite{Moi}, is no longer needed here. The analysis is carried out on the torus under the symmetry conditions described in \eqref{Equations primitives}, which differ from those studied by Titi and Li in \cite{Titi3}, \cite{Titi2}, \cite{Titi1}.

Our result can be compared to that of Gallagher in \cite{Gallagher}, who studied the 2D Navier-Stokes equations on the 3D torus with three components and obtained global existence for large data. In our case, for the 3D primitive equations with two components, we do not have global existence for large data.

We also explore another aspect of the primitive equations \eqref{Equations primitives}: demonstrating that they can be derived from the Navier-Stokes equations.

To proceed, let us first recall the formal derivation. Let us consider the anisotropic Navier-Stokes equations in a thin domain. \( \Omega_{1,\varepsilon} = (-1,1)^2 \times (-\varepsilon,\varepsilon) \) or \( \Omega_{2,\varepsilon} = \mathbb{R}^2 \times (-\varepsilon,\varepsilon) \):
\begin{eqnarray} \label{NS}
\left\{
\begin{array}{l}
\partial_t \tilde{u} + \tilde{u} \cdot \nabla \tilde{u} - \nu_h \Delta_h \tilde{u} - \nu_z \partial_z^2 \tilde{u} + \nabla \tilde{p} = 0, \\
\operatorname{div} \tilde{u} = 0,
\end{array}
\right.
\end{eqnarray}
with \( \nu_z = \varepsilon^\gamma \), \( \gamma > 2 \). Defining the rescaled unknowns
\[
\begin{aligned}
u_\varepsilon^h(x_h,x_v,t) &:= \tilde{u}^h(x_h, \varepsilon x_v, t), \\
u_\varepsilon^v(x_h,x_v,t) &:= \varepsilon^{-1} \tilde{u}^v(x_h, \varepsilon x_v, t), \\
u_\varepsilon &:= (u_\varepsilon^h, u_\varepsilon^v), \\
p_\varepsilon(x_h,x_v,t) &:= \tilde{p}(x_h, \varepsilon x_v, t),
\end{aligned}
\]
we rewrite \eqref{NS} as:
\begin{eqnarray} \label{NS remise à l'échelle}
\left\{
\begin{array}{l}
\partial_t u_\varepsilon^h + u_\varepsilon \cdot \nabla u_\varepsilon^h - \nu_h \Delta_h u_\varepsilon^h - \varepsilon^{\gamma-2} \partial_z^2 u_\varepsilon^h + \nabla_h p_\varepsilon = 0, \\
\varepsilon^2 \left( \partial_t u_\varepsilon^v + u_\varepsilon \cdot \nabla u_\varepsilon^v - \nu_h \Delta_h u_\varepsilon^v - \varepsilon^{\gamma-2} \partial_z^2 u_\varepsilon^v \right) + \partial_z p_\varepsilon = 0, \\
\operatorname{div} u_\varepsilon = 0, \\
u_\varepsilon^h \text{ even with respect to } z, \\
u_\varepsilon^v \text{ odd with respect to } z,
\end{array}
\right.
\end{eqnarray}
on the domain \( \Omega \) (independent of \( \varepsilon \)), with the same parity and periodicity conditions as in \eqref{Equations primitives}.

Formally, taking the limit \( \varepsilon \to 0 \) in \eqref{NS remise à l'échelle} yields the primitive equations \eqref{Equations primitives}.

On the 3D torus, this limit (for \( \gamma = 2 \) and primitive equations with vertical viscosity) was justified locally in time by Hieber et al. in \cite{Hieber2} using maximal parabolic regularity. In \cite{Moi}, we justified it globally in time in the same functional setting as the primitive equations, with less regular data.

For well-prepared initial data, Li, Titi, and Yuan studied this limit in \cite{Titi4}, proving convergence of Leray–Hopf weak solutions of \eqref{NS remise à l'échelle} to the unique strong solution of the primitive equations. However, for the well-posedness of the primitive equations, they refer to \cite{Titi5} and \cite{Titi2}, which deal with global strong solutions for arbitrary \( H^1 \) data in a different framework (involving temperature and incompatible boundary conditions with \( T = 0 \)).

In this article, we rigorously prove a convergence result in the sense of distributions between the strong solutions of the anisotropic Navier-Stokes equations \eqref{NS remise à l'échelle} (after establishing their well-posedness and uniqueness) and the constructed solutions of \eqref{Equations primitives}.

\subsection{Functional Spaces and notation}

We will use the anisotropic Besov space on the torus \( \Omega_1 \) introduced by Paicu in \cite{Paicu2}, defined by:
\[
\left\{ u \in \mathcal{S}'(\Omega_1) \ \middle| \ \|u\|_{\mathcal{B}^{0,1/2}} < +\infty \right\},
\]
where
\[
\displaylines{
\|u\|_{\mathcal{B}^{0,\frac{1}{2}}} := \left( \sum_{n' \in \mathbb{Z}^2} |\mathcal{F}u(n',0)|^2 \right)^{1/2} \hfill \cr
\hfill + \sum_{q \in \mathbb{N}} \left( \sum_{n' \in \mathbb{Z}^2} \sum_{2^q \leq |n_3| \leq 2^{q+1}} (1+|n_3|) |\mathcal{F}u(n)|^2 \right)^{1/2},
}
\]
with \( \mathcal{F}u(n) \) denoting the Fourier coefficients of \( u \).

An equivalent expression is given by:
\[
\|u\|_{\mathcal{B}^{0,1/2}} = \sum_{q \geq -1} 2^{q/2} \| \Delta_q^v u \|_{L^2(\Omega_1)} < +\infty,
\]
where the dyadic block \( \Delta_q^v \) is defined in Appendix \ref{Théorie de Littlewood-Paley anisotrope}.

We will also use the anisotropic Sobolev space \( H^{0,1/2} \), whose norm is defined as:
\[
\|u\|_{H^{0,1/2}}^2 := \sum_{n \in \mathbb{Z}^3} (1 + |n_3|) |\mathcal{F}u(n)|^2.
\]
Observing that this norm can also be written as:
\[
\|u\|_{H^{0,1/2}}^2 = \sum_{n' \in \mathbb{Z}^2} |\mathcal{F}u(n',0)|^2 + \sum_{q \in \mathbb{N}} \left( \sum_{n' \in \mathbb{Z}^2} \sum_{2^q \leq |n_3| < 2^{q+1}} (1+|n_3|) |\mathcal{F}u(n)|^2 \right),
\]
we see that it is equivalent to the following norm:
\begin{equation} \label{norme sobolev}
\|u\|_{H^{0,1/2}}^2 := \sum_{q \geq -1} \left( 2^{q/2} \|\Delta_q^v u\|_{L^2(\Omega_1)} \right)^2.
\end{equation}

More generally, we can define the anisotropic Sobolev space \( H^{0,s} \) with the norm:
\begin{equation} \label{norme sobolev2}
\|u\|_{H^{0,s}}^2 := \sum_{q \geq -1} \left( 2^{qs} \|\Delta_q^v u\|_{L^2(\Omega_1)} \right)^2.
\end{equation}

These definitions extend to the case \( \Omega_2 \); see Appendix \ref{Théorie de Littlewood-Paley anisotrope} for details.
\\
We also introduce time-dependent Besov-type spaces for vector fields, which incorporate Lebesgue-in-time integrability on the vertical dyadic blocks. We define:
\[
\|u\|_{\tilde{L}_T^\infty(\mathcal{B}^{0,1/2})} := \sum_{q \geq -1} 2^{q/2} \| \Delta_q^v u \|_{L_T^\infty(L^2)}.
\]

The space \( \tilde{L}_T^2(\mathcal{B}^{0,1/2}) \) is defined by the norm:
\[
\|u\|_{\tilde{L}_T^2(\mathcal{B}^{0,1/2})} := \sum_{q \geq -1} 2^{q/2} \| \Delta_q^v u \|_{L_T^2(L^2)}.
\]

Observe that, owing to Minskowski inequality:
\[
\|u\|_{L_T^\infty(\mathcal{B}^{0,1/2})} \leq \|u\|_{\tilde{L}_T^\infty(\mathcal{B}^{0,1/2})}, \quad
\|u\|_{L_T^2(\mathcal{B}^{0,1/2})} \leq \|u\|_{\tilde{L}_T^2(\mathcal{B}^{0,1/2})}.
\]

We refer to Appendix \ref{Théorie de Littlewood-Paley anisotrope} for more details and properties on these spaces.

Moreover, in order to simplify the demonstrations based on paradifferential calculus, we will note with $c_q$ a generic sequence which has a summable square root, i.e. $c_q\geq 0$ and $$\sum_{q\in\Z}\sqrt{c_q}\leq 1.$$ We will denote by $d_q$ a generic summable sequence such that $$\sum_{q\in\Z}d_q\leq 1.$$
\subsection{Main Results}

We now present the two theorems of global existence and uniqueness for the primitive equations \eqref{Equations primitives} that we prove in this article.

\begin{thm}\label{Caractère bien-posé équations primitives}
There exists a constant \( C \) such that for any divergence-free initial velocity \( u_0 = (u_0^h, u_0^v) \) satisfying \( u_0^h \in \mathcal{B}^{0,1/2}(\Omega) \) and decomposed as \( u_0^h = \overline{u}_0^h + \tilde{u}_0^h \), with \( \overline{u}_0^h \in L^2(\Omega_h) \) and \( \tilde{u}_0^h \in \mathcal{B}^{0,1/2}(\Omega) \), and satisfying the smallness condition:
\begin{equation}\label{hypothèse de petitesse}
\|\tilde{u}_0^h\|_{\mathcal{B}^{0,1/2}(\Omega)} \exp\left( \frac{\|\overline{u}_0^h\|_{L^2(\Omega_h)}}{C \nu_h} \right) \leq c \nu_h,
\end{equation}
then \eqref{Equations primitives} admits a unique global solution \( u = (u^h, u^v) \) such that:
\[
u^h \in \mathcal{C}_b(\mathbb{R}_+, \mathcal{B}^{0,1/2}(\Omega)), \quad \nabla_h u^h \in L^2(\mathbb{R}_+, \mathcal{B}^{0,1/2}(\Omega)).
\]
\end{thm}

\begin{remark}
The vertical component \( u^v \) is given explicitly by the formula:
\[
u^v = -\int_{-1}^z \operatorname{div}_h u^h \, dz'.
\]
\end{remark}

\begin{thm}\label{théorème unicité eq primitives}
The system \eqref{Equations primitives} has at most one solution in the space
\begin{multline}\label{tilde E définition}
\tilde{E} := \bigg\{ u = (u^h, u^v) \ \bigg| \ u^h \in \mathcal{C}_b(\R^+;H^{0,1/2}), \\ \nabla_h u^h \in L^2(\R^+;H^{0,1/2}) \bigg\}.
\end{multline}
\end{thm}
~
\\
We also prove an existence and uniqueness theorem for the scaled anisotropic Navier-Stokes system \eqref{NS remise à l'échelle}.

\begin{thm}\label{Caractère bien posé NS} Let \( \varepsilon > 0 \). There exists a constant \( C \) such that for any divergence-free initial data \( u_{\varepsilon,0} = (u_{\varepsilon,0}^h, u_{\varepsilon,0}^v) \) with \( u_{\varepsilon,0} \in \mathcal{B}^{0,1/2}(\Omega) \), decomposed as \( u_{\varepsilon,0} = \overline{u}_{\varepsilon,0} + \tilde{u}_{\varepsilon,0} \), where \( \overline{u}_{\varepsilon,0} \in L^2(\Omega_h) \) and \( \tilde{u}_{\varepsilon,0} \in \mathcal{B}^{0,1/2}(\Omega) \), and satisfying the smallness condition:
\begin{equation}\label{hypothèse de petitesse NS}
\|(\tilde{u}_{\varepsilon,0}^h, \varepsilon \tilde{u}_{\varepsilon,0}^v)\|_{\mathcal{B}^{0,1/2}(\Omega)} \exp\left( \frac{\|\overline{u}_{\varepsilon,0}\|_{L^2(\Omega_h)}}{C \nu_h} \right) \leq c \nu_h,
\end{equation}
then \eqref{NS remise à l'échelle} admits a unique global solution \( u_\varepsilon = (u_\varepsilon^h, u_\varepsilon^v) \) such that:
\[
u_\varepsilon \in \mathcal{C}_b(\mathbb{R}_+, \mathcal{B}^{0,1/2}(\Omega)), \quad \nabla u_\varepsilon \in L^2(\mathbb{R}_+, \mathcal{B}^{0,1/2}(\Omega)),
\]
and satisfies the following a priori estimate for all \( T > 0 \):
\begin{multline}\label{estimée a priori NS anisotrope}
\|(u_\varepsilon^h, \varepsilon u_\varepsilon^v)\|_{\tilde{L}_T^\infty(\mathcal{B}^{0,1/2})}
+ \sqrt{\nu_h} \|\nabla_h (u_\varepsilon^h, \varepsilon u_\varepsilon^v)\|_{\tilde{L}_T^2(\mathcal{B}^{0,1/2})} \\
+ \sqrt{\varepsilon^{\gamma - 2}} \|\partial_z (u_\varepsilon^h, \varepsilon u_\varepsilon^v)\|_{\tilde{L}_T^2(\mathcal{B}^{0,1/2})} \leq C_0,
\end{multline}
where \( C_0 \) is a constant depending only on the initial data.
\end{thm}

\begin{thm}\label{théorème unicité NS}
Let \( \varepsilon > 0 \) be fixed. Then system \eqref{NS remise à l'échelle} has at most one solution \( (u_\varepsilon^h, \varepsilon u_\varepsilon^v) \) in the space:
$$\displaylines{
\tilde{E}_\varepsilon := \bigg\{ u = (u^h, u^v) \ \bigg| \ (u^h, \varepsilon u^v) \in \mathcal{C}_b(\R^+;H^{0,1/2}), \hfill\cr\hfill \nabla_h (u^h, \varepsilon u^v) \in L^2(\R^+;H^{0,1/2}) \bigg\},}$$
uniformly in \( \varepsilon \).
\end{thm}

Finally, we state the main theorem of this article: the convergence of the solutions of the scaled anisotropic Navier-Stokes system \eqref{NS remise à l'échelle} to the unique solution of the primitive equations \eqref{Equations primitives}.

\begin{thm}\label{théorème de convergence}
Let \( (u_{\varepsilon,0})_\varepsilon \) be a sequence of initial data from Theorem~\ref{Caractère bien posé NS} converging in the sense of distributions to \( u_0 \), the initial data of Theorem~\ref{Caractère bien-posé équations primitives}. Then the corresponding sequence \( (u_\varepsilon)_\varepsilon \) of unique solutions to \eqref{NS remise à l'échelle} converges in the sense of distributions to \( u \), the unique solution of the primitive equations associated with \( u_0 \) from Theorem~\ref{Caractère bien-posé équations primitives}.
\end{thm}

\subsection{Sketch of the proofs}
We divide the proof of these results into three parts: the well-posedness and uniqueness of the primitive equations \eqref{Equations primitives}, those of the anisotropic Navier-Stokes equations, and the convergence of solutions between the two systems.

For the first part, we begin by focusing on a priori estimates in the space \( \mathcal{B}^{0,1/2} \). The existence theorem then follows easily by applying Friedrichs' method, as presented in \cite{BCD}. To obtain these estimates, we first observe that the vertical mean value of \( u \) satisfies the 2D Navier-Stokes equation with three components, for which global strong solutions with large data in \( H^1 \) are known, satisfying the classical energy equality of Navier-Stokes. We are then left to study the system without the vertical mean value. The most delicate term to estimate is the convection \( u \cdot \nabla u^h \). For this, we provide in Appendix \ref{convection} several so-called "convection lemmas" to handle this term, particularly because in the 3D torus domain \( \Omega_1 \), one must also account for horizontal mean values. These lemmas are strongly inspired by those found in Paicu's works \cite{Paicu1} and \cite{Paicu2}. 

The key idea behind the proof of these convection lemmas is to decompose the convection term into two parts:
\[
u^h \cdot \nabla u^h \quad \text{("the horizontal part")},
\]
and
\[
u^v \partial_z u^h \quad \text{("the vertical part")}.
\]
The horizontal part poses no difficulty since the needed control is provided by horizontal viscosity. Because we work in anisotropic Besov spaces, we simply use Bony's vertical paraproduct decomposition (recalled in the appendix) and estimate the terms in a standard way. The vertical part requires more effort since we have no direct information on \( u^v \), due to the absence of vertical viscosity. However, from the incompressibility condition, we gain \( L^2 \)-in-time control of \( \partial_z u^v = -\operatorname{div}_h u^h \). The idea is then to integrate by parts in the term \( \Delta_q^v (u^v \partial_z u^h) \Delta_q^v u^h \) to rewrite it as \( \Delta_q^v (\partial_z u^v u^h) \Delta_q^v u^h \). To justify this rigorously, Paicu took advantage of the decomposition \eqref{décomposition Chemin} introduced by Chemin in \cite{Chemin}.

Once these lemmas are established, it is straightforward to complete the a priori estimates for the system.

The uniqueness part of Theorem~\ref{théorème unicité eq primitives} relies on a double-logarithmic estimate of the form:
\[
\phi'(t) \leq f(t) \phi(t)(-\ln \phi(t)) \ln(-\ln \phi(t)),
\]
where \( \phi \) is a norm of the difference between two solutions in a suitable Banach space, and \( f \) is a locally integrable function. Osgood's lemma then yields uniqueness.

The reason "Why is uniqueness not proved directly in \( \mathcal{B}^{0,1/2} \)?" has been clearly explained by Paicu in \cite{Paicu1}. In fact, if we take \( u_1 \) and \( u_2 \) in \( \tilde{E} \) (defined in \eqref{tilde E définition}) as two solutions to \eqref{Equations primitives}, and denote \( w = u_2 - u_1 \), the goal is to show that \( w = 0 \) in \( \mathcal{D}'(\mathbb{R}_+^* \times \Omega) \). The equation satisfied by \( w^h \) contains nonlinear terms such as \( u_2 \cdot \nabla w^h \) and \( w \cdot \nabla u_1^h \). The "horizontal" terms \( u_2^h \cdot \nabla_h w^h \) and \( w^h \cdot \nabla_h u_1^h \) can be estimated classically via Bony's vertical decomposition. The term \( u_2^v \partial_z w^h \) is treated similarly to \( u^v \partial_z u^h \) in the a priori estimates using decomposition \eqref{décomposition Chemin}. The main difficulty lies in estimating \( w^v \partial_z u_1^h \). Vertically, \( w^v \) belongs to \( B_{2,1}^{1/2} \) via the vertical Poincaré inequality of Lemma \ref{Poincaré vertical} and the oddness with respect to the vertical variable, while \( \partial_z u_1^h \in B_{2,1}^{-1/2} \). Thus, their product naturally lies in \( B_{2,\infty}^{-1/2} \), suggesting the anisotropic candidate space to be \( \mathcal{B}_{2,\infty}^{0,-1/2} \), with the natural norm:
\[
\|w\|_{\mathcal{B}_{2,\infty}^{0,-1/2}} := \sup_q \left( 2^{-q/2} \|\Delta_q^v w\|_{L^2} \right).
\]
As noted in \cite{Paicu1}, this term is hard to estimate: it requires separate estimates for high and low frequencies, with low frequencies being especially problematic. This leads to a loss in the estimates and necessitates using Osgood's Lemma.

We then turn to the proofs of Theorems~\ref{Caractère bien posé NS} and \ref{théorème unicité NS}. The study of well-posedness and uniqueness for the anisotropic Navier-Stokes equations \eqref{NS remise à l'échelle} is quite similar to that for the primitive equations \eqref{Equations primitives}. Since \( (u_\varepsilon^h, \varepsilon u_\varepsilon^v) \) satisfies the standard Navier-Stokes equation (with anisotropic vertical viscosity), we consider it as the unknown and apply classical energy estimates in \( \mathcal{B}^{0,1/2} \), reusing the convection lemmas established earlier. Theorems~\ref{Caractère bien posé NS} and \ref{théorème unicité NS} then follow easily.

Finally, we prove, by a compactness argument, the convergence of strong solutions to the anisotropic Navier-Stokes equations \eqref{NS remise à l'échelle} (obtained in Theorem~\ref{Caractère bien posé NS}) to the unique solution of the primitive equations \eqref{Equations primitives} (obtained in Theorem~\ref{Caractère bien-posé équations primitives}), concluding with Theorem~\ref{théorème de convergence}.

\section{Primitive Equations}

This section is dedicated to the proof of the global existence theorem for the primitive equations.

To do so, we decompose equation \eqref{Equations primitives} into two coupled systems.

We set \( u^h = \tilde{u}^h + \overline{u}^h \), where \( \overline{u}^h \) satisfies:
\begin{equation} \label{Navier-Stokes 2D}
\left\{
\begin{array}{l}
\partial_t \overline{u}^h + \overline{u}^h \cdot \nabla_h \overline{u}^h - \nu_h \Delta_h \overline{u}^h + \nabla_h \overline{p} = 0, \\
\operatorname{div}_h \overline{u}^h = 0, \\
\overline{u}^h(t=0) = \overline{u}_0^h,
\end{array}
\right.
\end{equation}

and \( \tilde{u}^h \) satisfies:
\begin{equation} \label{Equations primitives sans moyenne verticale}
\left\{
\begin{array}{l}
\partial_t \tilde{u}^h + \tilde{u} \cdot \nabla \tilde{u}^h + \overline{u}^h \cdot \nabla_h \tilde{u}^h + \tilde{u}^h \cdot \nabla_h \overline{u}^h - \nu_h \Delta_h \tilde{u}^h + \nabla_h \tilde{p} = 0, \\
\partial_z \tilde{p} = 0, \\
\operatorname{div} \tilde{u} = 0, \\
\tilde{u}^h(t=0) = \tilde{u}_0^h.
\end{array}
\right.
\end{equation}
Equation \eqref{Navier-Stokes 2D} is a 2D Navier-Stokes type equation with three components. It is known that it has a unique global solution \( \overline{u}^h \) (see for instance \cite{Gallagher}) in the functional space:
\[
\overline{u}^h \in \mathcal{C}_b(\mathbb{R}_+; L^2(\Omega_h)) \cap L^2(\mathbb{R}_+; \dot{H}^1(\Omega_h)).
\]

Moreover, this solution satisfies the energy equality:
\begin{equation} \label{estimée NS 2D}
\|\overline{u}^h(t)\|_{L^2(\Omega_h)}^2 + 2 \nu_h \int_0^t \|\nabla_h \overline{u}^h(\tau)\|_{L^2(\Omega_h)}^2 \, d\tau = \|\overline{u}_0^h\|_{L^2}^2.
\end{equation}

\subsection{A priori estimate for the primitive equations without vertical mean value}
We now prove the a priori estimates for the system \eqref{Equations primitives sans moyenne verticale}.  
To do this, assume that we have a sufficiently regular solution \( \tilde{u} \) on \( [0,T] \times \Omega \).

Apply the operator \( \Delta_q^v \) to the first equation of \eqref{Equations primitives sans moyenne verticale} and take the scalar product with \( \Delta_q^v \tilde{u}^h \). Integrating by parts, we obtain:
\[
\begin{aligned}
\frac{1}{2} \frac{d}{dt} \|\Delta_q^v \tilde{u}^h\|_{L^2}^2 + \nu_h \|\nabla_h \tilde{u}^h\|_{L^2}^2 & \leq |(\Delta_q^v(\tilde{u} \cdot \nabla \tilde{u}^h) \mid \Delta_q^v \tilde{u}^h)_{L^2}| \\
& \quad + |(\Delta_q^v(\tilde{u}^h \cdot \nabla_h \overline{u}^h) \mid \Delta_q^v \tilde{u}^h)_{L^2}|,
\end{aligned}
\]
since, from the second and third equations in \eqref{Equations primitives sans moyenne verticale}, we obtain via integration by parts:
\[
\begin{aligned}
(\nabla_h \Delta_q^v \tilde{p} \mid \Delta_q^v \tilde{u}^h)_{L^2} &= -(\Delta_q^v \tilde{p} \mid \Delta_q^v \operatorname{div}_h \tilde{u}^h)_{L^2} \\
&= (\Delta_q^v \tilde{p} \mid \Delta_q^v \partial_z \tilde{u}^v)_{L^2} = (\partial_z \Delta_q^v \tilde{p} \mid \Delta_q^v \tilde{u}^v)_{L^2} = 0.
\end{aligned}
\]
Also, the term \( |(\Delta_q^v(\overline{u}^h \cdot \nabla_h \tilde{u}^h) \mid \Delta_q^v \tilde{u}^h)_{L^2}| \) vanishes since \( \operatorname{div}_h \overline{u}^h = 0 \) and \( \partial_z \overline{u}^h = 0 \).

Integrating in time over \( [0,t] \subset [0,T] \), we get:
$$\displaylines{\|\Delta_q^v \tilde{u}^h(t)\|_{L^2}^2+2\nu_h\int_0^t\|\nabla_h \tilde{u}^h(\tau)\|_{L^2}^2d\tau  \leq 2\int_0^t|(\Delta_q^v(\tilde{u}\cdot \nabla\tilde{u}^h)|\Delta_q^v \tilde{u}^h)_{L^2}|d\tau \cr\hfill+2\int_0^t|(\Delta_q^v(\tilde{u}^h\cdot\nabla_h\overline{u}^h)|\Delta_q^v \tilde{u}^h)_{L^2}|d\tau +\|\Delta_q^v \tilde{u}_0^h\|_{L^2}^2.}$$

As a consequence of Lemma~\ref{Premier lemme de convection} (for \( \Omega_2 \)) or Lemma~\ref{lemme convection final} (for the 3D torus \( \Omega_1 \)), we have:
\begin{multline} \label{terme de convection 1}
\int_0^T |(\Delta_q^v(\tilde{u} \cdot \nabla \tilde{u}^h) \mid \Delta_q^v \tilde{u}^h)_{L^2}| \, d\tau \leq C 2^{-q} c_q \|\tilde{u}^h\|_{\tilde{L}_T^\infty(\mathcal{B}^{0,1/2})}  \\ \times \|\nabla_h \tilde{u}^h\|_{\tilde{L}_T^2(\mathcal{B}^{0,1/2})}^2.
\end{multline}

By Lemma~\ref{3eme terme de convection}, we also have inequality \eqref{terme de convection2}.

Summing inequalities \eqref{terme de convection 1} and \eqref{terme de convection2}, and using the inequality \( ab \leq a^2 + b^2 \) for \( a,b \in \mathbb{R} \), we obtain for any \( t \in [0,T] \):
$$\displaylines{\|\Delta_q^v \tilde{u}^h(t)\|_{L^2}^2+2\nu_h\int_0^t \|\nabla_h \Delta_q^v \tilde{u}^h(\tau)\|_{L^2}^2 d\tau \hfill\cr\leq \|\Delta_q^v \tilde{u}_0^h\|_{L^2}^2+ 2^{-q}C c_q \|\tilde{u}^h\|_{\tilde{L}_T^\infty(\mathcal{B}^{0,1/2})}\|\nabla_h \tilde{u}^h\|_{\tilde{L}_T^2(\mathcal{B}^{0,1/2})}^2 \hfill\cr\hfill+2^{-q}c_q \frac{C^2}{\nu_h} \|\tilde{u}^h\|_{\tilde{L}_T^\infty(\mathcal{B}^{0,1/2})}^2 \|\nabla_h \overline{u}^h\|_{L_T^2(L^2(\Omega_h))}^2 + \nu_h 2^{-q}c_q \|\nabla_h \tilde{u}^h\|_{\tilde{L}_T^2(\mathcal{B}^{0,1/2})}^2.}$$

Taking the square root of the previous inequality, then taking the supremum over \( t \in [0,T] \), and summing over \( q \), we obtain:
\begin{multline} \label{étape de calcul}
\|\tilde{u}^h\|_{\tilde{L}_T^\infty(\mathcal{B}^{0,1/2})} + \sqrt{\nu_h} \|\nabla_h \tilde{u}^h\|_{\tilde{L}_T^2(\mathcal{B}^{0,1/2})} \\
\leq \|\tilde{u}_0^h\|_{\mathcal{B}^{0,1/2}} + C \|\tilde{u}^h\|_{\tilde{L}_T^\infty(\mathcal{B}^{0,1/2})}^{1/2} \|\nabla_h \tilde{u}^h\|_{\tilde{L}_T^2(\mathcal{B}^{0,1/2})} \\
+ \frac{C}{\sqrt{\nu_h}} \|\tilde{u}^h\|_{\tilde{L}_T^\infty(\mathcal{B}^{0,1/2})} \|\nabla_h \overline{u}^h\|_{L_T^2(L^2(\Omega_h))}.
\end{multline}

Now define the following field:
\begin{equation} \label{champ lambda}
\tilde{u}_\lambda(t,x) := e^{-\lambda \int_0^t \|\nabla_h \overline{u}^h(\tau)\|_{L^2(\Omega_h)}^2 \, d\tau} \tilde{u}(t,x),
\end{equation}
where \( \lambda > 0 \) is a parameter to be chosen later.

In particular, \( \tilde{u}_\lambda \) satisfies the equation:
\begin{multline} \label{Equations primitives sans moyenne verticale avec lambda}
\partial_t \tilde{u}_\lambda^h + \tilde{u} \cdot \nabla \tilde{u}_\lambda^h + \overline{u}^h \cdot \nabla_h \tilde{u}_\lambda^h + \tilde{u}_\lambda^h \cdot \nabla_h \overline{u}^h - \nu_h \Delta_h \tilde{u}_\lambda^h + \nabla_h \tilde{p}_\lambda \\
= -\lambda \|\nabla_h \overline{u}^h(t)\|_{L^2(\Omega_h)}^2 \tilde{u}_\lambda^h,
\end{multline}
where
\[
\tilde{p}_\lambda(t,x) := e^{-\lambda \int_0^t \|\nabla_h \overline{u}^h\|_{L^2(\Omega_h)}^2 \, d\tau} \tilde{p}(t,x).
\]

Repeating the previous estimates with \( \tilde{u}_\lambda^h \), we obtain:
\begin{multline} \label{estimée tilde u lambda}
\|\tilde{u}_\lambda^h\|_{\tilde{L}_T^\infty(\mathcal{B}^{0,1/2})} + \sqrt{\nu_h} \|\nabla_h \tilde{u}_\lambda^h\|_{\tilde{L}_T^2(\mathcal{B}^{0,1/2})} \\
\leq \|\tilde{u}_0^h\|_{\mathcal{B}^{0,1/2}} + C \|\tilde{u}^h\|_{\tilde{L}_T^\infty(\mathcal{B}^{0,1/2})}^{1/2} \|\nabla_h \tilde{u}_\lambda^h\|_{\tilde{L}_T^2(\mathcal{B}^{0,1/2})} \\
+ \left( \frac{C}{\sqrt{\nu_h}} - \sqrt{\lambda} \right) \|\tilde{u}_\lambda^h\|_{\tilde{L}_T^\infty(\mathcal{B}^{0,1/2})} \|\nabla_h \overline{u}^h\|_{L_T^2(L^2(\Omega_h))}.
\end{multline}

Let us now make the following smallness assumption:
\begin{equation} \label{hypothèse de petitesse2}
\|\tilde{u}^h\|_{\tilde{L}_T^\infty(\mathcal{B}^{0,1/2})} \leq c \nu_h \quad \text{with} \quad C c^{1/2} < \frac{1}{2\sqrt{2}}.
\end{equation}

Then the previous inequality gives:
$$\displaylines{\|\tilde{u}_\lambda^h\|_{\tilde{L}_T^\infty(\mathcal{B}^{0,1/2})} + \frac{\sqrt{\nu_h}}{2} \|\nabla_h \tilde{u}_\lambda^h\|_{\tilde{L}_T^2(\mathcal{B}^{0,1/2})}
\hfill\cr\hfill \leq \|\tilde{u}_0^h\|_{\mathcal{B}^{0,1/2}}+ \left( \frac{C}{\sqrt{\nu_h}} - \sqrt{\lambda} \right) \|\tilde{u}_\lambda^h\|_{\tilde{L}_T^\infty(\mathcal{B}^{0,1/2})} \|\nabla_h \overline{u}^h\|_{L_T^2(L^2(\Omega_h))}.}$$

Choosing \( \lambda \) large enough, e.g. \( \sqrt{\lambda} = \frac{2C}{\sqrt{\nu_h}} \), we get:
\[
\|\tilde{u}_\lambda^h\|_{\tilde{L}_T^\infty(\mathcal{B}^{0,1/2})} \leq \|\tilde{u}_0^h\|_{\mathcal{B}^{0,1/2}}.
\]

Using the definition \eqref{champ lambda}, the energy identity \eqref{estimée NS 2D}, and the smallness condition \eqref{hypothèse de petitesse}, we then obtain:
\[
\begin{aligned}
\|\tilde{u}^h\|_{\tilde{L}_T^\infty(\mathcal{B}^{0,1/2})}
& \leq \|\tilde{u}_\lambda^h\|_{\tilde{L}_T^\infty(\mathcal{B}^{0,1/2})} \exp\left( \frac{4C^2}{\nu_h} \|\nabla_h \overline{u}^h\|_{L_T^2(L^2(\Omega_h))}^2 \right) \\
& \leq \|\tilde{u}_0^h\|_{\mathcal{B}^{0,1/2}} \exp\left( \frac{2\sqrt{2} C^2}{\nu_h^{3/2}} \|\overline{u}_0^h\|_{L^2(\Omega_h)} \right) \\
& < c \nu_h.
\end{aligned}
\]

Let \( T^* \) be the maximal time such that \( \|\tilde{u}^h\|_{\tilde{L}_T^\infty(\mathcal{B}^{0,1/2})} \leq 2c\nu_h \) for all \( 0 < T < T^* \).  
By repeating the above estimates with this relaxed smallness assumption, and since \( t \mapsto \|\tilde{u}^h\|_{\tilde{L}_T^\infty(\mathcal{B}^{0,1/2})} \) is continuous, we conclude that \( T^* = +\infty \).  
Hence, we obtain the following a priori estimate:
\begin{equation} \label{estimée a priori}
\|\tilde{u}^h\|_{\tilde{L}^\infty(\mathcal{B}^{0,1/2})} + \frac{\sqrt{\nu_h}}{2} \|\nabla_h \tilde{u}^h\|_{\tilde{L}^2(\mathcal{B}^{0,1/2})} \leq c \nu_h.
\end{equation}

\subsection{Formal expression of the pressure}
We now prove the following lemma regarding the pressure:

\begin{lemma}\label{pression}
The pressure satisfies:
\begin{equation}\label{pression pour équations primitives}
\Delta p = -\frac{1}{2} \int_{-1}^1 \operatorname{div}_h \left[ \tilde{u} \cdot \nabla \tilde{u}^h + \overline{u}^h \cdot \nabla_h \tilde{u}^h + \tilde{u}^h \cdot \nabla_h \overline{u}^h \right] \, dz'.
\end{equation}
Hence, up to a constant, the pressure has the formal expression:
\[
p = Q(\tilde{u}^h, \tilde{u}^h) + Q(\overline{u}^h, \tilde{u}^h) + Q(\tilde{u}^h, \overline{u}^h),
\]
where \( Q(u, v) := \frac{1}{2} \int_{-1}^1 (-\Delta)^{-1} \operatorname{div} \operatorname{div}_h (u \otimes v) \, dz'. \)
\end{lemma}

\begin{proof}
Using the periodicity in the vertical variable, the incompressibility condition on \( u \), and the first equation of \eqref{Equations primitives}, we compute:
\begin{align*}
0 &= \partial_t \tilde{u}^v(x,y,1) - \Delta \tilde{u}^v(x,y,1) - \left[ \partial_t \tilde{u}^v(x,y,-1) - \Delta \tilde{u}^v(x,y,-1) \right] \\
&= \int_{-1}^1 \left( \partial_t \partial_z \tilde{u}^v - \Delta \partial_z \tilde{u}^v \right) \, dz' \\
&= - \int_{-1}^1 \operatorname{div}_h \left( \partial_t \tilde{u}^h - \Delta \tilde{u}^h \right) \, dz' \\
&= \int_{-1}^1 \operatorname{div}_h \left( \nabla_h p + \tilde{u} \cdot \nabla \tilde{u}^h + \overline{u}^h \cdot \nabla_h \tilde{u}^h + \tilde{u}^h \cdot \nabla_h \overline{u}^h \right) \, dz'.
\end{align*}

Thus, we obtain:
\[
\int_{-1}^1 \operatorname{div}_h (\nabla_h p) \, dz' = - \int_{-1}^1 \operatorname{div}_h \left( \tilde{u} \cdot \nabla \tilde{u}^h + \overline{u}^h \cdot \nabla_h \tilde{u}^h + \tilde{u}^h \cdot \nabla_h \overline{u}^h \right) \, dz'.
\]

Since \( \partial_z p = 0 \), it follows that:
\[
2\Delta p = - \int_{-1}^1 \operatorname{div}_h \left( \tilde{u} \cdot \nabla \tilde{u}^h + \overline{u}^h \cdot \nabla_h \tilde{u}^h + \tilde{u}^h \cdot \nabla_h \overline{u}^h \right) \, dz'.
\]

This proves the lemma.
\end{proof}

\subsection{Existence theorem for the system without vertical mean value}
To prove the existence of solutions to the system \eqref{Equations primitives}, we use Friedrichs' method (see Chapter 4 of \cite{BCD}).

We define the following spectral truncation operator:
\[
J_k u := \sum_{|m| \leq k} \mathcal{F}_h^{-1} \left( \mathbf{1}_{k^{-1} \leq |\xi_h| \leq k} \mathcal{F}_h u(\xi_h) \right)(x_h) \cdot \widehat{u}_m e^{i\pi m x_v},
\]
where \( \widehat{u}_m := \frac{1}{2} \int_{-1}^1 e^{-i\pi m x_v} u(x_h, x_v) \, dx_v \), and \( \mathcal{F}_h \) denotes the horizontal Fourier transform over \( \Omega_h \).

We consider the following approximate system:
\[
\left\{
\begin{array}{l}
\partial_t \tilde{u}_k^h + J_k(J_k \tilde{u}_k \cdot \nabla J_k \tilde{u}_k^h)
+ J_k(J_k \overline{u}^h \cdot \nabla_h J_k \tilde{u}_k^h)
+ J_k(J_k \tilde{u}_k^h \cdot \nabla_h J_k \overline{u}^h) \\
\quad - \Delta_h J_k \tilde{u}_k^h = J_k Q(J_k \tilde{u}_k^h, J_k \tilde{u}_k^h)
+ J_k Q(J_k \overline{u}^h, J_k \tilde{u}_k^h)
+ J_k Q(J_k \tilde{u}_k^h, J_k \overline{u}^h), \\
\operatorname{div} \tilde{u}_k = 0, \\
\tilde{u}_k(t=0) = J_k \tilde{u}_0,
\end{array}
\right.
\]
where \( \tilde{u}_k = (\tilde{u}_k^h, \tilde{u}_k^v) \) and the vertical component is given by the formal expression:
\[
\tilde{u}_k^v := -\int_{-1}^z \operatorname{div}_h(\tilde{u}_k^h) \, dz',
\]
which follows from \( \operatorname{div}_h \tilde{u}_k^h + \partial_z \tilde{u}_k^v = 0 \) and the oddness condition on \( \tilde{u}_k^v \). The operator \( Q \) is defined as in Lemma~\ref{pression}.

\begin{itemize}
\item[$\bullet$] By the Cauchy–Lipschitz theorem, this spectrally truncated system admits a unique maximal solution \( \tilde{u}_k^h \in \mathcal{C}^1([0, T_k[; L^2) \) with initial data \( J_k \tilde{u}_0^h \) for all \( k \in \mathbb{N} \).

\item[$\bullet$] Since \( J_k^2 = J_k \), it follows that \( J_k \tilde{u}_k^h \) also solves the system. By uniqueness, \( J_k \tilde{u}_k^h = \tilde{u}_k^h \), and thus \( \tilde{u}_k^h \) satisfies the system:
\[
\begin{aligned}
\partial_t \tilde{u}_k^h &+ J_k(\tilde{u}_k \cdot \nabla \tilde{u}_k^h)
+ J_k(J_k \overline{u}^h \cdot \nabla_h \tilde{u}_k^h)
+ J_k(\tilde{u}_k^h \cdot \nabla_h J_k \overline{u}^h) - \Delta \tilde{u}_k^h \\
&= J_k Q(\tilde{u}_k^h, \tilde{u}_k^h)
+ J_k Q(J_k \overline{u}^h, \tilde{u}_k^h)
+ J_k Q(\tilde{u}_k^h, J_k \overline{u}^h),
\end{aligned}
\]
with initial data \( J_k \tilde{u}_0^h \).

\item[$\bullet$] Performing the \( L^2 \) energy estimate and noting that the pressure terms vanish by integration by parts, and that
\[
(J_k(\tilde{u}_k \cdot \nabla \tilde{u}_k^h) \mid \tilde{u}_k^h)_{L^2} = 0, \quad
(J_k(\overline{u}_k^h \cdot \nabla_h \tilde{u}_k^h) \mid \tilde{u}_k^h)_{L^2} = 0,
\]
we obtain:
\[\displaylines{
\frac{1}{2} \frac{d}{dt} \|\tilde{u}_k^h(t)\|_{L^2}^2 + \|\nabla_h \tilde{u}_k^h(t)\|_{L^2}^2
\leq \|\nabla_h J_k \overline{u}^h\|_{L^\infty} \|\tilde{u}_k^h\|_{L^2}^2
\hfill\cr\hfill\leq C 2^{2k} \|\overline{u}_0\|_{L^2(\Omega_h)} \|\tilde{u}_k^h\|_{L^2}^2.}
\]

By Grönwall's lemma, \( \tilde{u}_k^h(t) \) remains locally bounded on the maximal interval of existence \( [0, T_k[ \), hence \( T_k = +\infty \) by continuation argument.
\end{itemize}

Using the a priori estimate \eqref{estimée a priori}, we deduce that the sequence of global solutions \( (\tilde{u}_k^h) \) is uniformly bounded in the space:
\[
\tilde{u}_k^h \in \tilde{L}^\infty(\mathbb{R}_+; \mathcal{B}^{0,1/2}), \quad
\nabla_h \tilde{u}_k^h \in \tilde{L}^2(\mathbb{R}_+; \mathcal{B}^{0,1/2}).
\]
In particular, \( (\tilde{u}_k^h)_{k \in \mathbb{N}} \) is bounded in \( L^\infty(\mathbb{R}_+; L^2) \), and from the equation satisfied by \( \tilde{u}_k^h \), we get that \( (\partial_t \tilde{u}_k^h) \) is bounded in \( L^\infty(\mathbb{R}_+; H^{-N}) \) for sufficiently large \( N \).

By Ascoli's theorem (up to extraction), \( \tilde{u}_k^h \to \tilde{u}^h \) in \( L^\infty_{\text{loc}}(\mathbb{R}_+; H_{\text{loc}}^{-N}) \). Since \( \tilde{u}_k^h \) is bounded in \( L^\infty(\mathbb{R}_+; L^2) \), interpolation gives:
\[
\tilde{u}_k^h \to \tilde{u}^h \quad \text{in } L^\infty_{\text{loc}}(\mathbb{R}_+; H_{\text{loc}}^{-s}) \quad \text{for all } s > 0.
\]

Since \( \tilde{u}_k^h \) is bounded in \( L^2_{\text{loc}}(H^\varepsilon) \) for all \( \varepsilon < \tfrac{1}{2} \), classical product laws in Sobolev spaces yield:
\[
\tilde{u}_k \otimes \tilde{u}_k^h \to \tilde{u} \otimes \tilde{u}^h \quad \text{in } L^2_{\text{loc}}(H_{\text{loc}}^{\varepsilon - \sigma - 3/2}),
\]
whenever \( \sigma < \varepsilon \). In particular, this convergence holds in the sense of distributions. Passing to the limit in the approximate system yields that \( \tilde{u} \) is a solution of \eqref{Equations primitives}.

Time continuity of the solution follows by standard arguments — see \cite{Paicu1} for further details.

\subsection{Uniqueness Lemma}

In this section, we aim to establish the following lemma, whose proof is strongly inspired by \cite{Paicu1}.

\begin{lemma}\label{lemme d'unicité}
Let \( u_1 \) and \( u_2 \) be two divergence-free vector fields such that \( u_1^h, u_2^h \in L_T^\infty(H^{0,1/2}) \), and \( \nabla_h u_1^h, \nabla_h u_2^h \in L_T^2(H^{0,1/2}) \). Let \( w \) be a divergence-free vector field such that \( w^h \in L_T^\infty(H^{0,1/2}) \) and \( \nabla_h w^h \in L_T^2(H^{0,1/2}) \), satisfying the equation:
\[
\left\{
\begin{array}{l}
\partial_t w^h + u_1 \cdot \nabla w^h + w \cdot \nabla u_2^h - \Delta_h w^h = -\nabla_h p, \\
\partial_z p = 0, \\
\operatorname{div} w = 0.
\end{array}
\right.
\]

Then, for every \( 0 < t < T \), the following estimate holds:
$$\displaylines{\|w^h(t)\|_{H^{0,-1/2}}^2 \leq \|w_0^h\|_{H^{0,-1/2}}^2+ \int_0^t C f(\tau)  \|w^h(t)  \|_{H^{0,-1/2}}^2 \hfill\cr\hfill \times \ln\left(1 + e + \frac{1}{\|w^h\|_{H^{0,1/2}}^2}\right) 
\ln\left( \ln\left(1 + e + \frac{1}{\|w^h\|_{H^{0,1/2}}^2}\right)\right)d\tau,}$$
where \( f \) is a locally integrable function in time given by:
\begin{equation}\label{fonction unicité}
\begin{aligned}
f(t) := \left(1 + \|u_1^h(t)\|_{H^{0,1/2}}^2 + \|u_2^h(t)\|_{H^{0,1/2}}^2 + \|w^h(t)\|_{H^{0,1/2}}^2 \right) \\
\times \left(1 + \|\nabla_h u_1^h(t)\|_{H^{0,1/2}}^2 + \|\nabla_h u_2^h(t)\|_{H^{0,1/2}}^2 + \|\nabla_h w^h(t)\|_{H^{0,1/2}}^2 \right).
\end{aligned}
\end{equation}
\end{lemma}

The proof is nearly identical to the one given by Paicu in \cite{Paicu1} for \( \Omega_2 \) and \cite{Paicu2} for \( \Omega_1 \). We provide here the main steps of the argument in the case \( \Omega_2 \), emphasizing the slight differences from \cite{Paicu1}. For a full proof, we refer to Lemma 2.3.12 in \cite{Thèse}.

\begin{proof}
By taking the energy estimate on the equation satisfied by \( w^h \) using the norm \eqref{norme sobolev2}, we have:
$$\displaylines{\frac{1}{2}\frac{d}{dt}\|w^h(t)\|_{H^{0,-1/2}}^2+\nu_h\|\nabla_h w^h(t)\|_{H^{0,-1/2}}^2 \leq \sum_{q\geq -1}2^{-q}\int \Delta_q^v (u_1\cdot\nabla w^h)\Delta_q^v w^h dx \hfill\cr\hfill +\sum_{q\geq -1}2^{-q}\int \Delta_q^v(w\cdot\nabla u_2^h)\Delta_q^v w^h dx.}$$

To handle the horizontal terms, i.e., those involving only derivatives in horizontal variables, we use the following lemma:

\begin{lemma}\label{Premier sous-lemme d'unicité}
Let \( a: \Omega \to \mathbb{R}^2 \) and \( b: \Omega\to \R \) such that \( b \in L_T^\infty(H^{0,1/2}) \), and \( \nabla_h a, \nabla_h b \in L_T^2(H^{0,1/2}) \). Then, there exists a constant \( C \) such that for all \( t \in (0,T) \),
$$\displaylines{\sum_{q\geq -1}2^{-q}\int \Delta_q^v(a\cdot\nabla_h b)\Delta_q^v b \ dx\leq Cg(t)(1-\ln \|b\|_{H^{0,-1/2}}^2)\ln(1-\ln(\|b\|_{H^{0,-1/2}}^2))\hfill\cr\hfill+\frac{\nu_h}{20}\|\nabla_h b\|_{H^{0,-1/2}}^2,}$$ where $$g(t)\mathrel{\mathop:}=(1+\|a(t)\|_{H^{0,1/2}}^2+\|b(t)\|_{H^{0,1/2}}^2)(1+\|\nabla_h a(t)\|_{H^{0,1/2}}^2 +\|\nabla_h b(t)\|_{H^{0,1/2}}^2).$$
\end{lemma}

The proof can be found between pages 215 and 220 of \cite{Paicu1}, replacing \( u_h \) by \( a \) and \( w \) by \( b \).

In our case, choosing \( a = u_1^h \) and \( b = w^h \) (componentwise), we deduce:
\begin{multline}\label{premier terme unicité}
\sum_{q \geq -1} 2^{-q} \int \Delta_q^v(u_1^h \cdot \nabla_h w^h) \Delta_q^v w^h \, dx \\
\leq C g(t) \|w^h\|_{H^{0,-1/2}}^2 \left(1 - \ln \|w^h\|_{H^{0,-1/2}}^2 \right) \left(1 + \ln(1 - \ln \|w^h\|_{H^{0,-1/2}}^2)\right) \\
+ \frac{\nu_h}{20} \|\nabla_h w^h\|_{H^{0,-1/2}}^2.
\end{multline}

Let us now estimate the other horizontal terms that contain only horizontal derivatives.

\begin{lemma}\label{Deuxième sous-lemme d'unicité}
Let \( a:\Omega\to \R^2 \) and \( b:\Omega\to \R^2 \) belong to the space \( L_T^\infty(H^{0,1/2}) \), and assume that \( \nabla_h a \) and \( \nabla_h b \) belong to \( L_T^2(H^{0,1/2}) \). Then there exists a constant \( C \) such that for all \( t\in ]0,T[ \), we have:
$$\displaylines{\sum_{q\geq -1}2^{-q}\int \Delta_q^v(b \cdot\nabla_h a)\Delta_q^v b dx \\ \leq Cg(t)(1-\ln \|w^h\|_{H^{0,-1/2}}^2)\|w^h\|_{H^{0,-1/2}}^2 \hfill\cr\hfill +\frac{\nu_h}{20}\|\nabla_h w^h\|_{H^{0,-1/2}}^2,}$$ where $$g(t)\mathrel{\mathop:}=(1+\|a(t)\|_{H^{0,1/2}}^2+\|b(t)\|_{H^{0,1/2}}^2)(1+\|\nabla_h a(t)\|_{H^{0,1/2}}^2 +\|\nabla_h b(t)\|_{H^{0,1/2}}^2).$$
\end{lemma}

The proof of this lemma is given between pages 220 and 224 in \cite{Paicu1}, replacing \( v \) by \( a \) and \( w^h \) by \( b \).

Taking \( a = u_2^h \) and \( b = w^h \), we obtain:
\begin{multline}\label{deuxième terme unicité}
\sum_{q\geq -1}2^{-q}\sum_{|q-q'|\leq 4}\int \Delta_q^v(w^h \cdot\nabla_h u_2^h)\Delta_q^v w^h\, dx \\
\leq C f(t)(1-\ln \|w^h\|_{H^{0,-1/2}}^2)\|w^h\|_{H^{0,-1/2}}^2 + \frac{\nu_h}{20}\|\nabla_h w^h\|_{H^{0,-1/2}}^2.
\end{multline}

We next study the terms involving vertical derivatives. We will transform vertical derivatives into horizontal derivatives using the relation \( \partial_z u^v = -\dive_h u^h \). We must estimate the following two terms:
$$I = \sum_{q\geq -1}2^{-q}\int \Delta_q^v(u_1^v \partial_z w^h)\Delta_q^v w^h\, dx,$$ and $$\tilde{I} = \sum_{q\geq -1}2^{-q}\int \Delta_q^v(w^v \partial_z u_2^h)\Delta_q^v w^h\, dx.$$

To estimate the first term, we use the following lemma:

\begin{lemma}\label{Troisième sous-lemme d'unicité}
Let \( a \) be a three-dimensional divergence-free vector field and \( b \) a scalar function in \( L_T^\infty(H^{0,1/2}) \), such that \( \nabla_h a \) and \( \nabla_h b \) belong to \( L_T^2(H^{0,1/2}) \). Then there exists a constant \( C \) such that for all \( t \in (0,T) \),
\[
\sum_{q \geq -1} 2^{-q} \int \Delta_q^v(a^v \partial_z b) \Delta_q^v b \, dx \leq C g(t)(1 - \ln \|b\|_{H^{0,-1/2}}^2) + \frac{2\nu_h}{25} \|\nabla_h b\|_{H^{0,-1/2}}^2,
\]
where
\[
g(t) := (1 + \|a^h(t)\|_{H^{0,1/2}}^2 + \|b(t)\|_{H^{0,1/2}}^2)(1 + \|\nabla_h a^h(t)\|_{H^{0,1/2}}^2 + \|\nabla_h b(t)\|_{H^{0,1/2}}^2).
\]
\end{lemma}

The proof of this lemma is given between pages 225 and 226 of \cite{Paicu1}, replacing \( u_1 \) with \( a \) and \( w^h \) with \( b \).

Taking \( a = u_1 \), \( b = w^h \), we get:
\begin{multline}\label{troisième terme unicité}
I := \sum_{q \geq -1} 2^{-q} \int \Delta_q^v(u_1^v \partial_z w^h) \Delta_q^v w^h \, dx \\
\leq C(1 - \ln \|w^h(t)\|_{H^{0,-1/2}}^2) f(t) \|w^h(t)\|_{H^{0,-1/2}}^2 + \frac{2\nu_h}{25} \|\nabla_h w^h(t)\|_{H^{0,-1/2}}^2.
\end{multline}

Finally, we estimate the term
\[
\tilde{I} := \sum_{q \geq -1} 2^{-q} \int \Delta_q^v(w^v \partial_z u_2^h) \Delta_q^v w^h \, dx.
\]

\begin{lemma}\label{Quatrième sous-lemme d'unicité}
Let \( a: \Omega \to \mathbb{R}_h^2 \) and \( b \) be a three-dimensional divergence-free vector field such that \( b \in L_T^\infty(H^{0,1/2}) \), and \( \nabla_h a \), \( \nabla_h b \in L_T^2(H^{0,1/2}) \). Then there exists a constant \( C \) such that on $[0,T]$, we have a.e.
$$\displaylines{\sum_{q\geq -1}2^{-q}\int \Delta_q^v(b^v \partial_z a)\Delta_q^v b^h dx \hfill\cr\hfill \leq Cg(1-\ln \|w^h\|_{H^{0,-1/2}}^2)\ln(1-\ln \|w^h\|_{H^{0,-1/2}}^2)\|w^h\|_{H^{0,-1/2}}^2 \hfill\cr\hfill +\frac{3\nu_h}{25}\|\nabla_h w^h\|_{H^{0,-1/2}}^2,}$$ where $$g(t)\mathrel{\mathop:}=(1+\|a(t)\|_{H^{0,1/2}}^2+\|b^h(t)\|_{H^{0,1/2}}^2)(1+\|\nabla_h a(t)\|_{H^{0,1/2}}^2 +\|\nabla_h b^h(t)\|_{H^{0,1/2}}^2).$$
\end{lemma}

\begin{proof}
We set the time $t\in [0,T]$ once and for all, and it will not be mentioned again hereafter for the sake of simplicity.

We use Bony's decomposition with respect to the vertical variable (see Appendix~\ref{Théorie de Littlewood-Paley anisotrope}), which gives:
\[
\tilde{I} = \tilde{T}_1^v + \tilde{T}_2^v + \tilde{R}^v,
\]
where
\[
\begin{aligned}
\tilde{T}_1^v &= \sum_{q \geq -1} 2^{-q} \sum_{|q - q'| \leq 4} \int \Delta_q^v(S_{q'-1}^v b^v \partial_z \Delta_{q'}^v a) \Delta_q^v b^h \, dx, \\
\tilde{T}_2^v &= \sum_{q \geq -1} 2^{-q} \sum_{|q - q'| \leq 4} \int \Delta_q^v(S_{q'-1}^v \partial_z a \cdot \Delta_{q'}^v b^v) \Delta_q^v b^h \, dx, \\
\tilde{R}^v &= \sum_{q \geq -1} 2^{-q} \sum_{q' \geq q - 1} \int \Delta_q^v(\Delta_{q'}^v b^v \cdot \Delta_{q'}^v \partial_z a) \Delta_q^v b^h \, dx.
\end{aligned}
\]

The proof is essentially identical to that in \cite{Paicu1}, pages 227–231, with \( a \) replaced by \( v \) and \( b \) by \( w^h \).

For the term \(\tilde{T}_1^v\), we have
$$\displaylines{\tilde{T}_1^v=\sum_{q\geq -1}2^{-q}\sum_{|q-q'|\leq 4}\sum_{1\leq j\leq q'}\int \Delta_q^v(\Delta_j^v b^v \partial_z \Delta_{q'}^v a)\Delta_q^v b^hdx \hfill\cr\hfill \leq\sum_{q\geq -1}2^{-q}\sum_{|q-q'|\leq 4}\sum_{j}\int |\Delta_q^v(\Delta_j^v b^v \partial_z \Delta_{q'}^v a)\Delta_q^v b^h|dx.}$$

We have the following inequality on the term \(\tilde{T}_1^v\):
\[
\tilde{T}_1^v \leq \tilde{T}_{1,0}^v + \tilde{T}_{1,N}^v + \tilde{T}_1^{v,N},
\]
where $N$ (depending on $t$) will be defined later and
\[
\begin{cases}
\displaystyle \tilde{T}_{1,0}^v := \sum_q 2^{-q} \sum_{|q - q'| \leq 4} \int \left| S_0^v b^v \, \partial_z \Delta_{q'}^v a \, \Delta_q^v b^h \right| dx, \\
\displaystyle \tilde{T}_{1,N}^v := \sum_q 2^{-q} \sum_{|q - q'| \leq 4} \sum_{0 \leq j \leq N} \int \left| \Delta_j^v b^v \, \partial_z \Delta_{q'}^v a \, \Delta_q^v b^h \right| dx, \\
\displaystyle \tilde{T}_1^{v,N} := \sum_q 2^{-q} \sum_{|q - q'| \leq 4} \sum_{j \geq N} \int \left| \Delta_j^v b^v \, \partial_z \Delta_{q'}^v a \, \Delta_q^v b^h \right| dx.
\end{cases}
\]

Compared to \cite{Paicu1}, the only difference will lie in the estimate of \(\tilde{T}_{1,0}^v\). Let us then estimate this term.

By noticing, thanks to Bernstein's lemma \ref{Bernstein et Gagliardo-Nirenberg}, that
\[
\|S_0^v b^v\|_{L_v^\infty L_h^2} = \|\Delta_{-1}^v b^v\|_{L_v^\infty L_h^2} \leq C \|\Delta_{-1}^v b^v\|_{L^2} \leq C \|b^v\|_{H^{0,-1/2}},
\]
and
\[
\|\partial_z \Delta_{q'}^v a\|_{L_v^2 L_h^4} \leq C 2^{q'} \|\Delta_{q'}^v a\|_{L_v^2 L_h^4},
\]
we then have by the Gagliardo-Nirenberg inequality \eqref{inégalité de Gagliardo-Nirenberg généralisée}:
\begin{align*}
\tilde{T}_{1,0}^v &\leq \sum_q 2^{-q} \sum_{|q - q'| \leq 4} \|S_0^v b^v\|_{L_v^\infty L_h^2} \|\partial_z \Delta_{q'}^v a\|_{L_v^2 L_h^4} \|\Delta_q^v b^h\|_{L_v^2 L_h^4} \\
&\leq \sum_q 2^{-q} \sum_{|q - q'| \leq 4} C 2^{q'} \|b^v\|_{H^{0,-1/2}} \|\Delta_{q'}^v a\|_{L^2}^{1/2} \|\nabla_h \Delta_{q'}^v a\|_{L^2}^{1/2} \\
&\quad \times \big( \|\Delta_q^v b^h\|_{L^2}^{1/2} \|\nabla_h \Delta_q^v b^h\|_{L^2}^{1/2} + \|\Delta_q^v b^h\|_{L^2} \big).
\end{align*}

Since \(b^v\) is odd with respect to the vertical variable, by the "vertical" Poincaré-Wirtinger lemma \ref{Poincaré vertical} and the divergence-free condition, we have
$$\displaylines{\|b^v\|_{H^{0,-1/2}}=\sqrt{\sum_{q\geq -1}(2^{q/2}\|\Delta_q^v b^v\|_{L^2})^2}\leq C \sqrt{\sum_{q\geq -1}(2^{q/2}\|\Delta_q^v \partial_z b^v\|_{L^2})^2} \hfill\cr\hfill \leq C \|\partial_z b^v\|_{H^{0,-1/2}} \leq C \|\nabla_h b^h\|_{H^{0,-1/2}}.}$$

Taking into account that \(u\) belongs to \(H^{0,1/2}\) and that \(w\) is estimated in \(H^{0,-1/2}\), we obtain
\begin{align*}
\tilde{T}_{1,0}^v &\leq C \|a\|_{H^{0,1/2}}^{1/2} \|\nabla_h a\|_{H^{0,1/2}}^{1/2} \|\nabla_h b^h\|_{H^{0,-1/2}} \big( \|b^h\|_{H^{0,-1/2}}^{1/2} \|\nabla_h b^h\|_{H^{0,-1/2}}^{1/2} \\
&\quad + \|b^h\|_{H^{0,-1/2}} \big) \\
&\leq C \big( \|a\|_{H^{0,1/2}}^{2} \|\nabla_h a\|_{H^{0,-1/2}}^{2} + f(t) \big) \|b^h\|_{H^{0,-1/2}}^{2} + \frac{\nu_h}{50} \|\nabla_h b^h\|_{H^{0,-1/2}}^{2}.
\end{align*}

Hence we obtain
\begin{equation}\label{ineq_tilde_t_1_0}
\tilde{T}_{1,0}^v \leq C f(t) \|b^h\|_{H^{0,-1/2}}^{2} + \frac{\nu_h}{50} \|\nabla_h b^h\|_{H^{0,-1/2}}^{2}.
\end{equation}

We now estimate $\tilde{T}_{1,N}^v$. Let $\varepsilon<1/2$ be a parameter that will be chosen at the end. 
\begin{align*}
\tilde{T}_{1,N}^v & \leq \sum_{q\geq -1}\sum_{|q-q'|\leq 4} 2^{-q}\sum_{0\leq j\leq N}\|\Delta_q^v(\Delta_j^v b^v \partial_z \Delta_{q'}^v a)\|_{L_v^2 L_h^{\frac{2}{1+\varepsilon}}} \|\Delta_q^v b^h\|_{L_v^2 L_h^{\frac{2}{1-\varepsilon}}} \\
& \leq \sum_{q\geq -1}\sum_{|q-q'|\leq 4}2^{-q}\sum_{0\leq j \leq N} \|\Delta_j^v b^v\|_{L_v^\infty L_h^2}\|\partial_z \Delta_{q'}^v a\|_{L_v^2 L_h^{2/\varepsilon}}\|\Delta_q^v b^h\|_{L_v^2 L_h^{\frac{2}{1-\varepsilon}}}.
\end{align*}

Using Lemma \ref{Inégalité Gagliardo-Nirenberg avec mooyenne}, we obtain
$$
\displaylines{
\tilde{T}_{1,N}^v \leq \frac{C}{\sqrt{\varepsilon}}\sum_{q\geq -1}\sum_{|q-q'|\leq 4} 2^{-q}\sum_{0\leq j \leq N} \|\Delta_j^v b^v\|_{L_v^\infty L_h^2} \Big( \|\partial_z \Delta_{q'}^v a\|_{L^2}^\varepsilon \|\partial_z \Delta_{q'}^v \nabla_h a\|_{L^2}^{1-\varepsilon}
\hfill\cr\hfill
+ \|\partial_z \Delta_{q'}^v a\|_{L^2} \Big) \Big( \|\Delta_{q}^v b^h\|_{L^2}^\varepsilon \|\Delta_{q}^v \nabla_h b^h\|_{L^2}^{1-\varepsilon} + \|\Delta_{q'}^v b^h\|_{L^2} \Big).
}
$$

On the one hand, from the Bernstein lemma \ref{Bernstein et Gagliardo-Nirenberg} and the assumption $a\in H^{0,1/2}$, we get
\begin{align*}
\|\partial_z \Delta_{q'}^v a\|_{L^2}^\varepsilon \|\partial_z \Delta_{q'}^v \nabla_h a\|_{L^2}^{1-\varepsilon} & \leq C2^{q'} \|\Delta_{q'}^v a\|_{L^2}^\varepsilon \|\Delta_{q'}^v \nabla_h a\|_{L^2}^{1-\varepsilon} \\
& \leq C 2^{q'/2} c_{q'} \|a\|_{H^{0,1/2}}^\varepsilon \|\nabla_h a\|_{H^{0,1/2}}^{1-\varepsilon},
\end{align*}
where $(c_{q'})$ is a sequence such that $\sum_q \tilde{c}_q^2=1$.

On the other hand, since $b^h$ is estimated in $H^{0,-1/2}$, we have
$$
\|\Delta_q^v b^h\|_{L^2}^{1-\varepsilon}\|\Delta_q^v \nabla_h b^h\|_{L^2}^\varepsilon \leq C 2^{q/2} \tilde{c}_q \|b^h\|_{H^{0,-1/2}}^{1/2} \|\nabla_h b^h\|_{H^{0,-1/2}}^{1/2}.
$$

It remains to estimate $\sum_{0\leq j \leq N} \|\Delta_j^v b^v\|_{L_v^\infty L_h^2}$. From the Bernstein Lemma \ref{Bernstein et Gagliardo-Nirenberg} and the divergence-free condition, we get
\begin{align*}
\sum_{0\leq j \leq N} \|\Delta_j^v b^v\|_{L_v^\infty L_h^2}
& \leq C \sum_{0\leq j \leq N} 2^{j/2} \|\Delta_j^v b^v\|_{L^2} \\
& \leq C \sum_{0\leq j \leq N} 2^{-j/2} \|\Delta_j^v \partial_z b^v\|_{L^2} \\
& = C \sum_{0\leq j \leq N} 2^{-j/2} \|\Delta_j^v \nabla_h b^v\|_{L^2} \\
& \leq C \sqrt{N} \|\nabla_h b^h\|_{H^{0,-1/2}},
\end{align*}
where we used the Cauchy-Schwarz inequality in the last step.

Thus, we obtain the estimate
$$
\displaylines{
\tilde{T}_{1,N}^v \leq C \sqrt{\frac{N}{\varepsilon}} \sum_{q\geq -1} \tilde{c}_q \left( \sum_{|q-q'|\leq 4} c_{q'} \right) \|a\|_{H^{0,1/2}}^\varepsilon \|\nabla_h a\|_{H^{0,1/2}}^{1-\varepsilon}
\hfill\cr\hfill
\times \left( \|b^h\|_{H^{0,-1/2}}^{1-\varepsilon} \|\nabla_h b^h\|_{H^{0,-1/2}}^{1+\varepsilon} + \|b^h\|_{H^{0,-1/2}} \|\nabla_h b^h\|_{H^{0,-1/2}} \right)
\cr
\leq C \sqrt{\frac{N}{\varepsilon}} \|a\|_{H^{0,1/2}}^\varepsilon \|\nabla_h a\|_{H^{0,1/2}}^{1-\varepsilon} \bigg( \|b^h\|_{H^{0,-1/2}}^{1-\varepsilon} \|\nabla_h b^h\|_{H^{0,-1/2}}^{1+\varepsilon} \hfill\cr\hfill + \|b^h\|_{H^{0,-1/2}} \|\nabla_h b^h\|_{H^{0,-1/2}} \bigg).
}
$$

Using the inequality $ab \leq a^{\frac{2}{1+\varepsilon}} + b^{\frac{2}{1-\varepsilon}}$, we obtain
$$
\displaylines{
\tilde{T}_{1,N}^v \leq c \left(\frac{N}{\varepsilon}\right)^{\frac{1}{1-\varepsilon}} \|a\|_{H^{0,1/2}}^{\frac{\varepsilon}{1-\varepsilon}} \|\nabla_h a\|_{H^{0,1/2}}^2
+ c \frac{N}{\varepsilon} \|a\|_{H^{0,1/2}}^{2\varepsilon} \|\nabla_h a\|_{H^{0,1/2}}^{2(1-\varepsilon)} \|b^h\|_{H^{0,1/2}}^2
\hfill\cr\hfill
+ \frac{\nu_h}{100} \|\nabla_h b^h\|_{H^{0,-1/2}}^2.
}
$$

Choosing $\varepsilon = \frac{1}{1+\ln N}$ gives
\begin{equation}\label{estimée tilde{T}_{1,N}^v}
\tilde{T}_{1,N}^{v} \leq C(1+\ln N)f(t)\|b^h(t)\|_{H^{0,-1/2}}^2+\frac{\nu_h}{100}\|\nabla_h b^h\|_{H^{0,-1/2}}^2.
\end{equation}

We now estimate $\tilde{T}_1^{v,N}$. From Hölder’s inequality, we have
$$
\tilde{T}_{1}^{v,N}\leq \sum_{q\geq -1} 2^{-q} \sum_{|q-q'|\leq 4} \sum_{j\geq N} \|\Delta_j^v b^v\|_{L_v^\infty L_h^2} \|\partial_z \Delta_{q'}^v a\|_{L_v^2 L_h^4} \|\Delta_q^v b^h\|_{L_v^2 L_h^4}.
$$

Since $\nabla_h b^h \in H^{0,1/2}$, we get
\begin{align*}
\sum_{j\geq N}\|\Delta_j^v b^v\|_{L_v^\infty L_h^2}
& \leq C \sum_{j\geq N} 2^{j/2} \|\Delta_j^v b^v\|_{L^2} \\
& \leq C \sum_{j\geq N} 2^{-j/2} \|\Delta_j^v \partial_z b^v\|_{L^2} \\
& \leq C \sum_{j\geq N} 2^{-j} 2^{j/2} \|\Delta_j^v \nabla_h b^h\|_{L^2} \\
& \leq C 2^{-N} \|\nabla_h b^h\|_{H^{0,1/2}}.
\end{align*}

Thus, we obtain
$$
\displaylines{
\tilde{T}_1^{v,N} \leq C 2^{-N} \|\nabla_h b^h\|_{H^{0,1/2}} \sum_{q\geq -1} 2^{-q} \sum_{|q-q'|\leq 4} 2^{q'} \|\Delta_{q'}^v a\|_{L^2}^{1/2} \|\Delta_{q'}^v \nabla_h a\|_{L^2}^{1/2}
\hfill\cr\hfill
\times \left( \|\Delta_q^v b^h\|_{L^2}^{1/2} \|\nabla_h \Delta_q^v b^h\|_{L^2}^{1/2} + \|\Delta_q^v b^h\|_{L^2} \right)
\cr
\leq C 2^{-N} \|\nabla_h b^h\|_{H^{0,1/2}} \|a\|_{H^{0,1/2}}^{1/2} \|\nabla_h a\|_{H^{0,1/2}}^{1/2}
\bigg( \|b^h\|_{H^{0,-1/2}}^{1/2} \|\nabla_h b^h\|_{H^{0,-1/2}}^{1/2} \hfill\cr\hfill + \|b^h\|_{H^{0,-1/2}} \bigg).
}
$$

Using the inequality $ab \leq \frac{3}{4} a^{4/3} + \frac{1}{4} b^4$, we obtain
\begin{multline}\label{tilde{T}_1^{v,N}}
\tilde{T}_1^{v,N} \leq C2^{-4N/3}f(t)\|b^h(t)\|_{H^{0,-1/2}}^{2/3}
+ Cf(t)\|b^h(t)\|_{H^{0,-1/2}}^2
\hfill\cr\hfill
+ \frac{\nu_h}{100}\|\nabla_h b^h(t)\|_{H^{0,-1/2}}^2.
\end{multline}

With the choice $N = -\ln \|b^h(t)\|_{H^{0,-1/2}}^2$, and using estimates \eqref{ineq_tilde_t_1_0}, \eqref{estimée tilde{T}_{1,N}^v} and \eqref{tilde{T}_1^{v,N}}, we finally obtain
\begin{equation}\label{estimée tilde T1v}
    \begin{aligned}
        \tilde{T}_1^v\leq Cf(t) \|b^h(t)\|_{H^{0,-1/2}}^2 (1-\ln \|b^h(t)\|_{H^{0,-1/2}}^2) \\
        \times \ln(1-\ln \|b^h(t)\|_{H^{0,-1/2}}^2)
        + \frac{2\nu_h}{25} \|\nabla_h b^h(t)\|_{H^{0,-1/2}}^2.
    \end{aligned}
\end{equation}

The other terms $\tilde{T}_2^v$ and $\tilde{R}^v$ are estimated almost identically as in \cite{Paicu1}.
\end{proof}

Eventually, taking \( a = u_2^h \), \( b = w \), we obtain a.e. on $[0,T]$:
\begin{multline}\label{quatrième terme unicité}
\tilde{I} = \sum_{q \geq -1} 2^{-q} \int \Delta_q^v(w^v \partial_z u_2^h) \Delta_q^v w^h \, dx \\
\leq C f(t) \|w^h\|_{H^{0,-1/2}}^2 (1 - \ln \|w^h\|_{H^{0,-1/2}}^2) \ln(1 - \ln \|w^h\|_{H^{0,-1/2}}^2) \\
+ \frac{3\nu_h}{25} \|\nabla_h w^h\|_{H^{0,-1/2}}^2.
\end{multline}

Finally, summing estimates \eqref{premier terme unicité}, \eqref{deuxième terme unicité}, \eqref{troisième terme unicité}, and \eqref{quatrième terme unicité}, and letting \( \phi(t) := \|w^h(t)\|_{H^{0,-1/2}}^2 \), we obtain:
\[
\phi(t) \leq \phi(0)+ C \int_0^t f(\tau) \phi(\tau) (1 - \ln \phi(\tau)) \ln(1 - \ln \phi(\tau))d\tau,
\]
with \( f \in L_{\text{loc}}^1 \) as defined in \eqref{fonction unicité}.
\end{proof}
\subsection{Uniqueness}

The proof of uniqueness is a direct application of Lemma \ref{lemme d'unicité}. We present here the uniqueness result for the primitive equations in the space \( H^{0,1/2} \), which is a larger space than the anisotropic Besov space \( \mathcal{B}^{0,1/2} \). Let \( u_1^h \) and \( u_2^h \) be two solutions in \( L_{\text{loc}}^\infty(H^{0,1/2}) \) with \( \nabla_h u_1^h \), \( \nabla_h u_2^h \in L_{\text{loc}}^2(H^{0,1/2}) \), corresponding to the same initial data.

We aim to show that \( w^h = u_2^h - u_1^h \) vanishes in \( \mathcal{C}([0,T];H^{0,-1/2}) \) and satisfies \( \nabla_h w^h = 0 \) in \( L_T^2(H^{0,-1/2}) \). The equation satisfied by \( w \) is:
\[
\left\{
\begin{array}{l}
\partial_t w^h + u_1 \cdot \nabla w^h + w \cdot \nabla u_2^h - \Delta_h w^h = -\nabla_h p, \\
\partial_z p = 0, \\
\dive w = 0.
\end{array}
\right.
\]

Lemma \ref{lemme d'unicité} yields the estimate 
$$\displaylines{\|w^h(t)\|_{H^{0,-1/2}}^2\leq \|w_0^h\|_{H^{0,-1/2}}^2+ \int_0^t C f(\tau) \|w^h\|_{H^{0,-1/2}}^2 (1-\ln \|w^h\|_{H^{0,-1/2}}^2)\hfill\cr\hfill\times\ln(1-\ln \|w^h\|_{H^{0,-1/2}}^2)d\tau,}$$
where \( f \) is the function such that
\[
\begin{aligned}
\|f\|_{L_T^1} \leq C_T (1 + \|\nabla_h u_1^h\|_{L_T^2(H^{0,1/2})}^2 + \|\nabla_h u_2^h\|_{L_T^2(H^{0,1/2})}^2 + \|\nabla_h w^h\|_{L_T^2(H^{0,1/2})}^2) \\
\times (1 + \|u_1^h\|_{L_T^\infty(H^{0,1/2})}^2 + \|u_2^h\|_{L_T^\infty(H^{0,1/2})}^2 + \|w^h\|_{L_T^\infty(H^{0,1/2})}^2).
\end{aligned}
\]

Note that, since \( w^h = u_2^h - u_1^h \), we have
\[
\|w^h\|_{L_T^\infty(H^{0,1/2})} \leq \|u_1^h\|_{L_T^\infty(H^{0,1/2})} + \|u_2^h\|_{L_T^\infty(H^{0,1/2})},
\]
and likewise
\[
\|\nabla_h w^h\|_{L_T^2(H^{0,1/2})} \leq \|\nabla_h u_1^h\|_{L_T^2(H^{0,1/2})} + \|\nabla_h u_2^h\|_{L_T^2(H^{0,1/2})}.
\]
Thus, $f$ is indeed integrable on $[0,T]$.

Applying Osgood's Lemma \ref{Osgood} with \( \rho(s) := \|w^h(s)\|_{H^{0,-1/2}}^2 \) and \( \mu(r) := r(1 - \ln r)\ln(1 - \ln r) \), we conclude the proof of uniqueness.

Moreover, one can derive the following continuity estimate with respect to the initial data:
$$\|u_1^h(t)-u_2^h(t)\|_{H^{0,-1/2}}^2 \leq C \exp\left(-(-\ln \|u_1^h(0)-u_2^h(0)\|_{H^{0,-1/2}}^2)^{\exp(\int_0^t f(s)ds)}\right).$$

\section{Anisotropic Navier-Stokes equations}
Let us prove the existence theorem for primitive equations.

\subsection{Rewriting and decomposing the equations}~\\
First of all, we rewrite \eqref{NS remise à l'échelle} in the following way:
\begin{equation}\label{NS remise à l'échelle2}
    \left\{\begin{array}{l} \displaystyle
     \frac{d}{dt}\begin{pmatrix}
         u_\varepsilon^h \\ \varepsilon u_\varepsilon^v
     \end{pmatrix} +u_\varepsilon\cdot\nabla\begin{pmatrix}
         u_\varepsilon^h \\ \varepsilon u_\varepsilon^v
     \end{pmatrix}-\nu_h \Delta_h \begin{pmatrix}
         u_\varepsilon^h \\ \varepsilon u_\varepsilon^v
     \end{pmatrix} -\varepsilon^{\gamma-2}\partial_z^2 \begin{pmatrix}
         u_\varepsilon^h \\ \varepsilon u_\varepsilon^v
     \end{pmatrix} \\ ~~~~~~~~~~~~~~~~~~~~~~~~~~~~~~~~~~~~~~~~~~~~~~~~~~~~~~~~~~~~~~~~~~~~~ +\nabla_\varepsilon p_\varepsilon=0,\\
     \dive_\varepsilon (u_\varepsilon^h,\varepsilon u_\varepsilon^v)=0, \\
     u_{\varepsilon}^h \ \text{even with respect to the vertical variable} \ z, \\
     u_{\varepsilon}^v \ \text{odd with respect to the vertical variable} \ z,
\end{array}\right. \end{equation}
where $\dive_\varepsilon$ is defined by : \begin{eqnarray}\label{divergence epsilon}
\dive_\varepsilon U\mathrel{\mathop:}=\dive_h(U^h)+\varepsilon^{-1}\partial_z U^v
\end{eqnarray} and $\displaystyle \nabla_\varepsilon$ by \begin{eqnarray}\label{nabla epsilon}
\nabla_\varepsilon\mathrel{\mathop:}=\begin{pmatrix}
    \nabla_h \\ \varepsilon^{-1} \partial_z \end{pmatrix}. \end{eqnarray}

Let us now decompose this system into two coupled systems.

We set $u_\varepsilon=\tilde{u}_\varepsilon+\overline{u}_\varepsilon$ verifying
\begin{equation}\label{Navier-Stokes 2D epsilon}\left\{ \begin{array}{l}  
\partial_t \overline{u}_\varepsilon^h+\overline{u}_\varepsilon^h\cdot \nabla_h \overline{u}_\varepsilon^h-\nu_h\Delta_h \overline{u}_\varepsilon^h +\nabla_h \overline{p}_\varepsilon=0, \\ \dive_h \overline{u}_\varepsilon^h=0, \\ \overline{u}_\varepsilon^h(t=0)=\overline{u}_{\varepsilon,0}^h, \\ \overline{u}_\varepsilon^v=0, \end{array}\right.\end{equation}

and 
\begin{equation}\label{NS sans moyenne verticale}\left\{ \begin{array}{l}  
\displaystyle\frac{d}{dt}\begin{pmatrix} \tilde{u}_\varepsilon^h \\ \varepsilon \tilde{u}_\varepsilon^v \end{pmatrix} +\tilde{u}_\varepsilon\cdot \nabla \begin{pmatrix} \tilde{u}_\varepsilon^h \\ \varepsilon \tilde{u}_\varepsilon^v \end{pmatrix}+\overline{u}_\varepsilon^h\cdot \nabla_h \begin{pmatrix} \tilde{u}_\varepsilon^h \\ \varepsilon \tilde{u}_\varepsilon^v \end{pmatrix} +\tilde{u}_\varepsilon^h\cdot\nabla_h \begin{pmatrix} \overline{u}_\varepsilon^h \\ 0 \end{pmatrix} \\ ~~~~~~~~~~~~~~~~~~~~~~~~~~ -\nu_h\Delta_h \begin{pmatrix} \tilde{u}_\varepsilon^h \\ \varepsilon \tilde{u}_\varepsilon^v \end{pmatrix}-\varepsilon^{\gamma-2}\partial_z^2 \begin{pmatrix}
         u_\varepsilon^h \\ \varepsilon u_\varepsilon^v
     \end{pmatrix}+\nabla_h \tilde{p}_\varepsilon=0 
\\  \dive_\varepsilon (u_\varepsilon^h,\varepsilon u_\varepsilon^v)=0, 
\\ \\ \begin{pmatrix} \tilde{u}_\varepsilon^h \\ \varepsilon \tilde{u}_\varepsilon^v \end{pmatrix}(t=0)=\begin{pmatrix} \tilde{u}_{\varepsilon,0}^h \\ \varepsilon \tilde{u}_{\varepsilon,0}^v \end{pmatrix}.  \end{array}\right.\end{equation}
\subsection{Bidimensional Navier-Stokes equations}~\\
The equation \eqref{Navier-Stokes 2D epsilon} is a bidimensional Navier-Stokes equation with three components. We then know that it has a unique global solution $\overline{u}_\varepsilon^h$ (see for instance \cite{Gallagher}) in the functional space $$\overline{u}_\varepsilon^h\in L^\infty(\R_+;L^2(\Omega_h))\cap L^2(\R_+; \dot H^1(\Omega_h)),$$ 
verifying the following energy equality \begin{equation}\label{estimée NS 2D epsilon}\|\overline{u}_\varepsilon^h\|_{L^2(\Omega_h)}+2\nu_h\int_0^t \|\nabla_h \overline{u}_\varepsilon^h(\tau)\|_{L^2(\Omega_h)}^2 d\tau=\|\overline{u}_{\varepsilon,0}^h\|_{L^2}^2.\end{equation}

\subsection{A priori estimate of the Navier-Stokes equations without vertical mean value}~\\
Let us now prove the a priori estimates for the equation \eqref{NS sans moyenne verticale}.

First, let us remove the pressure. In the case of the Navier-Stokes equation without anisotropy, we apply the Leray projector $\mathbb{P}$. With the anisotropy present here, we consider the anisotropic Leray projector: \begin{equation}\label{projecteur de Leray epsilon}\mathbb{P}_\varepsilon\mathrel{\mathop:}=Id+\nabla_\varepsilon (-\Delta_\varepsilon)^{-1}\dive_{\varepsilon}, \quad  \text{where} \ \Delta_\varepsilon\mathrel{\mathop:}= \dive_\varepsilon \nabla_\varepsilon,\end{equation}
with $\dive_\varepsilon$ and $\nabla_\varepsilon$ defined in \eqref{divergence epsilon} and \eqref{nabla epsilon}.

 \begin{lemma}{\cite{Moi}}\label{continuité opérateur de Leray} The operator $-\nabla_\varepsilon(-\Delta_\varepsilon)^{-1}\dive_\varepsilon$ is an orthogonal projector on $L^2$.
\end{lemma}
In particular, $\mathbb{P}_\varepsilon$ is a continuous operator of $\mathcal{B}^{0,1/2}$ with norm 1 that verifies \begin{equation}\label{identité projecteur}\mathbb{P}_\varepsilon (u_{\varepsilon}^h,\varepsilon u_\varepsilon^v)= (u_{\varepsilon}^h,\varepsilon u_\varepsilon^v),\end{equation} for $u_\varepsilon\in \mathcal{B}^{0,1/2}$ verifying $\dive_\varepsilon u=0$. 
\\
Finding solution $\left((\tilde{u}_\varepsilon^h,\varepsilon \tilde{u}_\varepsilon^v),p_\varepsilon\right)$ to the system \eqref{NS sans moyenne verticale} with initial data $\begin{pmatrix}
    \tilde{u}_{\varepsilon,0}^h \\ \varepsilon \tilde{u}_{\varepsilon,0}^v
\end{pmatrix}$ is equivalent to finding solutions $(\tilde{u}_\varepsilon^h,\varepsilon \tilde{u}_\varepsilon^v)$  to the following system : \begin{equation}\label{NS projeté epsilon}\left\{ \begin{array}{l}  
\displaystyle\frac{d}{dt}\begin{pmatrix} \tilde{u}_\varepsilon^h \\ \varepsilon \tilde{u}_\varepsilon^v \end{pmatrix} +\mathbb{P}_\varepsilon\left(\tilde{u}_\varepsilon\cdot \nabla \begin{pmatrix} \tilde{u}_\varepsilon^h \\ \varepsilon \tilde{u}_\varepsilon^v \end{pmatrix}\right)+\mathbb{P}_\varepsilon
\left(\overline{u}_\varepsilon^h\cdot \nabla_h \begin{pmatrix} \tilde{u}_\varepsilon^h \\ \varepsilon \tilde{u}_\varepsilon^v \end{pmatrix}\right)  \\ ~~~~~~+\mathbb{P}_\varepsilon\left(\tilde{u}_\varepsilon^h\cdot\nabla_h \begin{pmatrix} \overline{u}_\varepsilon^h \\ 0 \end{pmatrix}\right) -\nu_h\Delta_h \begin{pmatrix} \tilde{u}_\varepsilon^h \\ \varepsilon \tilde{u}_\varepsilon^v \end{pmatrix}-\varepsilon^{\gamma-2}\partial_z^2 \begin{pmatrix}
         u_\varepsilon^h \\ \varepsilon u_\varepsilon^v
     \end{pmatrix}=0, 
\\ \begin{pmatrix} \tilde{u}_\varepsilon^h \\ \varepsilon \tilde{u}_\varepsilon^v \end{pmatrix}(t=0)=\begin{pmatrix} \tilde{u}_{\varepsilon,0}^h \\ \varepsilon \tilde{u}_{\varepsilon,0}^v \end{pmatrix}.  \end{array}\right.\end{equation}

Let us assume we have a smooth enough solution $\begin{pmatrix} \tilde{u}_\varepsilon^h \\ \varepsilon \tilde{u}_\varepsilon^v \end{pmatrix}$ on $[0,T]\times \Omega$.

By applying the localization operator $\Delta_q^v$ to the system \eqref{NS projeté epsilon}, taking the product scalar with $\Delta_q^v \begin{pmatrix} \tilde{u}_\varepsilon^h \\ \varepsilon \tilde{u}_\varepsilon^v \end{pmatrix}$ and integrating by parts, we obtain the following equality:
$$\displaylines{\frac{1}{2}\frac{d}{dt}\|\Delta_q^v\begin{pmatrix}
    \tilde{u}_\varepsilon^h \\ \varepsilon \tilde{u}_\varepsilon^v
\end{pmatrix}\|_{L^2}^2+\nu_h \|\Delta_q^v \nabla_h \begin{pmatrix}
    \tilde{u}_\varepsilon^h \\ \varepsilon \tilde{u}_\varepsilon^v
\end{pmatrix}\|_{L^2}^2+\varepsilon^{\gamma-2}\|\Delta_q^v \partial_z \begin{pmatrix}
    \tilde{u}_\varepsilon^h \\ \varepsilon \tilde{u}_\varepsilon^v
\end{pmatrix}\|_{L^2}^2 \hfill\cr\hfill = -\left(\mathbb{P}_\varepsilon\Delta_q^v \left(\tilde{u}_\varepsilon\cdot\nabla \begin{pmatrix}
    \tilde{u}_\varepsilon^h \\ \varepsilon \tilde{u}_\varepsilon^v
\end{pmatrix}\right) \middle| \Delta_q^v \begin{pmatrix}
    \tilde{u}_\varepsilon^h \\ \varepsilon \tilde{u}_\varepsilon^v
\end{pmatrix} \right)_{L^2}
\hfill\cr\hfill
-\left(\mathbb{P}_\varepsilon\Delta_q^v \left(\overline{u}_\varepsilon\cdot\nabla \begin{pmatrix}
    \tilde{u}_\varepsilon^h \\ \varepsilon \tilde{u}_\varepsilon^v
\end{pmatrix}\right) \middle| \Delta_q^v \begin{pmatrix}
    \tilde{u}_\varepsilon^h \\ \varepsilon \tilde{u}_\varepsilon^v
\end{pmatrix} \right)_{L^2} 
\cr\hfill
-\left(\mathbb{P}_\varepsilon\Delta_q^v \left(\tilde{u}_\varepsilon\cdot\nabla \begin{pmatrix}
    \overline{u}_\varepsilon^h \\ 0
\end{pmatrix}\right) \middle| \Delta_q^v \begin{pmatrix}
    \tilde{u}_\varepsilon^h \\ \varepsilon \tilde{u}_\varepsilon^v
\end{pmatrix} \right)_{L^2}.}$$

By orthogonality of $\mathbb{P}_\varepsilon$ (Lemma \ref{projecteur de Leray epsilon}) and equality \eqref{identité projecteur}, we have by integrating the above identity between $0$ and $t$ with $t\in [0,T]$ :

\begin{multline}\label{première étape estimée a priori NS}
    \|\Delta_q^v \begin{pmatrix}
    \tilde{u}_\varepsilon^h \\ \varepsilon \tilde{u}_\varepsilon^v
\end{pmatrix}(t)\|_{L^2}^2+2\nu_h\int_0^t \|\Delta_q^v \nabla_h \begin{pmatrix}
    \tilde{u}_\varepsilon^h \\ \varepsilon \tilde{u}_\varepsilon^v
\end{pmatrix}\|_{L^2}^2 d\tau \\ +2\varepsilon^{\gamma-2}\int_0^t \|\Delta_q^v \partial_z \begin{pmatrix}
    \tilde{u}_\varepsilon^h \\ \varepsilon \tilde{u}_\varepsilon^v
\end{pmatrix}\|_{L^2}^2 d\tau \\ \leq \|\Delta_q^v \begin{pmatrix}
    \tilde{u}_{\varepsilon,0}^h \\ \varepsilon \tilde{u}_{\varepsilon,0}^v
\end{pmatrix}\|_{L^2}^2 +2\int_0^t \left|\left(\Delta_q^v \left(\tilde{u}_\varepsilon\cdot\nabla \begin{pmatrix}
    \tilde{u}_\varepsilon^h \\ \varepsilon \tilde{u}_\varepsilon^v
\end{pmatrix}\right) \middle| \Delta_q^v \begin{pmatrix}
    \tilde{u}_\varepsilon^h \\ \varepsilon \tilde{u}_\varepsilon^v
\end{pmatrix} \right)_{L^2}\right| d\tau \\ +2\int_0^t \left|\left(\Delta_q^v \left(\overline{u}_\varepsilon\cdot\nabla \begin{pmatrix}
    \tilde{u}_\varepsilon^h \\ \varepsilon \tilde{u}_\varepsilon^v
\end{pmatrix}\right) \middle| \Delta_q^v \begin{pmatrix}
    \tilde{u}_\varepsilon^h \\ \varepsilon \tilde{u}_\varepsilon^v
\end{pmatrix} \right)_{L^2} \right| d\tau \\ +2\int_0^t \left|\left(\Delta_q^v \left(\tilde{u}_\varepsilon\cdot\nabla 
    \overline{u}_\varepsilon^h\right) \middle| \Delta_q^v
    \tilde{u}_\varepsilon^h \right)_{L^2}\right| d\tau.
\end{multline}

By Lemma \ref{Premier lemme de convection} (in the case $\Omega_2$) or Lemma \ref{lemme convection final} (in the case $\Omega_1$), we have the following inequality for the first non-linear term:
$$\displaylines{\int_0^t \left|\left(\Delta_q^v \left(\tilde{u}_\varepsilon\cdot\nabla \begin{pmatrix}
    \tilde{u}_\varepsilon^h \\ \varepsilon \tilde{u}_\varepsilon^v
\end{pmatrix}\right) \middle| \Delta_q^v \begin{pmatrix}
    \tilde{u}_\varepsilon^h \\ \varepsilon \tilde{u}_\varepsilon^v
\end{pmatrix} \right)_{L^2}\right| d\tau \hfill\cr\leq C c_q 2^{-q}\bigg(\|\nabla_h \tilde{u}_\varepsilon^h\|_{\tilde{L}_T^2(\mathcal{B}^{0,1/2})}\|\begin{pmatrix}
    \tilde{u}_\varepsilon^h \\ \varepsilon \tilde{u}_\varepsilon^v
\end{pmatrix}\|_{\tilde{L}_t^\infty(\mathcal{B}^{0,1/2})} \hfill\cr\hfill\times \|\nabla_h \begin{pmatrix}
    \tilde{u}_\varepsilon^h \\ \varepsilon \tilde{u}_\varepsilon^v
\end{pmatrix}\|_{\tilde{L}_T^2(\mathcal{B}^{0,1/2})}+\|\tilde{u}_\varepsilon^h\|_{\tilde{L}_t^\infty(\mathcal{B}^{0,1/2})}\|\nabla_h \begin{pmatrix}
    \tilde{u}_\varepsilon^h \\ \varepsilon \tilde{u}_\varepsilon^v
\end{pmatrix}\|_{\tilde{L}_T^2(\mathcal{B}^{0,1/2})}^2\bigg)
\cr \leq C c_q 2^{-q} \|\begin{pmatrix}
    \tilde{u}_\varepsilon^h \\ \varepsilon \tilde{u}_\varepsilon^v
\end{pmatrix}\|_{\tilde{L}_T^\infty(\mathcal{B}^{0,1/2})} \|\nabla_h \begin{pmatrix}
    \tilde{u}_\varepsilon^h \\ \varepsilon \tilde{u}_\varepsilon^v
\end{pmatrix}\|_{\tilde{L}_T^2(\mathcal{B}^{0,1/2})}^2. \hfill }$$

Since $\overline{u}_\varepsilon^h$ is constant with respect to the vertical variable, we have :
$$\displaylines{\int_0^t \left|\left(\Delta_q^v \left(\overline{u}_\varepsilon\cdot\nabla \begin{pmatrix}
    \tilde{u}_\varepsilon^h \\ \varepsilon \tilde{u}_\varepsilon^v
\end{pmatrix}\right) \middle| \Delta_q^v \begin{pmatrix}
    \tilde{u}_\varepsilon^h \\ \varepsilon \tilde{u}_\varepsilon^v
\end{pmatrix} \right)_{L^2} \right| d\tau \hfill\cr\hfill = \int_0^t \left|\left( \overline{u}_\varepsilon\cdot\nabla \Delta_q^v \begin{pmatrix}
    \tilde{u}_\varepsilon^h \\ \varepsilon \tilde{u}_\varepsilon^v
\end{pmatrix} \middle| \Delta_q^v \begin{pmatrix}
    \tilde{u}_\varepsilon^h \\ \varepsilon \tilde{u}_\varepsilon^v
\end{pmatrix} \right)_{L^2} \right| d\tau=0}$$
where the last equality is obtained by integration by parts and the fact that $\dive_h \overline{u}_\varepsilon^h=0$.

By Lemma \ref{3eme terme de convection}, we get the following inequality: 
$$\displaylines{\int_0^t \left|\left(\mathbb{P}_\varepsilon\Delta_q^v \left(\tilde{u}_\varepsilon\cdot\nabla 
    \overline{u}_\varepsilon^h \right) \middle| \Delta_q^v
    \tilde{u}_\varepsilon^h \right)_{L^2}\right| d\tau \hfill\cr\hfill \leq C 2^{-q} c_q \|\nabla_h \overline{u}_\varepsilon^h\|_{L_T^2(L^2(\Omega_h))}\|
    \tilde{u}_\varepsilon^h\|_{\tilde{L}_T^\infty(\mathcal{B}^{0,1/2})} \|\nabla_h \tilde{u}_\varepsilon^h\|_{\tilde{L}_T^2(\mathcal{B}^{0,1/2})}. }$$

By inequality \eqref{première étape estimée a priori NS}, we then obtain:
$$\displaylines{\|\Delta_q^v \begin{pmatrix}
    \tilde{u}_\varepsilon^h \\ \varepsilon \tilde{u}_\varepsilon^v
\end{pmatrix}(t)\|_{L^2}^2+2\nu_h\int_0^t \|\Delta_q^v \nabla_h \begin{pmatrix}
    \tilde{u}_\varepsilon^h \\ \varepsilon \tilde{u}_\varepsilon^v
\end{pmatrix}\|_{L^2}^2 d\tau \\ \hfill\cr\hfill +2\varepsilon^{\gamma-2}\int_0^t \|\Delta_q^v \partial_z \begin{pmatrix}
    \tilde{u}_\varepsilon^h \\ \varepsilon \tilde{u}_\varepsilon^v
\end{pmatrix}\|_{L^2}^2 d\tau \cr \leq \|\Delta_q^v \begin{pmatrix}
    \tilde{u}_{\varepsilon,0}^h \\ \varepsilon \tilde{u}_{\varepsilon,0}^v
\end{pmatrix}\|_{L^2}^2+ C c_q 2^{-q} \|\begin{pmatrix}
    \tilde{u}_\varepsilon^h \\ \varepsilon \tilde{u}_\varepsilon^v
\end{pmatrix}\|_{\tilde{L}_T^\infty(\mathcal{B}^{0,1/2})} \|\nabla_h \begin{pmatrix}
    \tilde{u}_\varepsilon^h \\ \varepsilon \tilde{u}_\varepsilon^v
\end{pmatrix}\|_{\tilde{L}_T^2(\mathcal{B}^{0,1/2})}^2 \cr\hfill +C 2^{-q} c_q \|\nabla_h \overline{u}_\varepsilon^h\|_{L_T^2(L^2(\Omega_h))}\|\begin{pmatrix}
    \tilde{u}_\varepsilon^h \\ \varepsilon \tilde{u}_\varepsilon^v
\end{pmatrix}\|_{\tilde{L}_T^\infty(\mathcal{B}^{0,1/2})} \|\nabla_h \begin{pmatrix}
    \tilde{u}_\varepsilon^h \\ \varepsilon \tilde{u}_\varepsilon^v
\end{pmatrix}\|_{\tilde{L}_T^2(\mathcal{B}^{0,1/2})}.}$$

By inequality $2ab\leq a^2+b^2$, we then have:  
$$\displaylines{\|\Delta_q^v \begin{pmatrix}
    \tilde{u}_\varepsilon^h \\ \varepsilon \tilde{u}_\varepsilon^v
\end{pmatrix}(t)\|_{L^2}^2+2\nu_h\int_0^t \|\Delta_q^v \nabla_h \begin{pmatrix}
    \tilde{u}_\varepsilon^h \\ \varepsilon \tilde{u}_\varepsilon^v
\end{pmatrix}\|_{L^2}^2 d\tau \\ \hfill\cr\hfill +2\varepsilon^{\gamma-2}\int_0^t \|\Delta_q^v \partial_z \begin{pmatrix}
    \tilde{u}_\varepsilon^h \\ \varepsilon \tilde{u}_\varepsilon^v
\end{pmatrix}\|_{L^2}^2 d\tau \cr \leq \|\Delta_q^v \begin{pmatrix}
    \tilde{u}_{\varepsilon,0}^h \\ \varepsilon \tilde{u}_{\varepsilon,0}^v
\end{pmatrix}\|_{L^2}^2+ C c_q 2^{-q} \|\begin{pmatrix}
    \tilde{u}_\varepsilon^h \\ \varepsilon \tilde{u}_\varepsilon^v
\end{pmatrix}\|_{\tilde{L}_T^\infty(\mathcal{B}^{0,1/2})} \|\nabla_h \begin{pmatrix}
    \tilde{u}_\varepsilon^h \\ \varepsilon \tilde{u}_\varepsilon^v
\end{pmatrix}\|_{\tilde{L}_T^2(\mathcal{B}^{0,1/2})}^2 \cr\hfill +2^{-q}c_q\nu_h \|\nabla_h \begin{pmatrix}
    \tilde{u}_\varepsilon^h \\ \varepsilon \tilde{u}_\varepsilon^v
\end{pmatrix}\|_{\tilde{L}_T^2(\mathcal{B}^{0,1/2})}^2  \cr\hfill+\frac{C^2}{\nu_h} 2^{-q} c_q\|\begin{pmatrix}
    \tilde{u}_\varepsilon^h \\ \varepsilon \tilde{u}_\varepsilon^v
\end{pmatrix}\|_{\tilde{L}_T^\infty(\mathcal{B}^{0,1/2})}^2 \|\nabla_h \overline{u}_\varepsilon^h\|_{L_T^2(L^2(\Omega_h))}^2.}$$

Taking the square root of the previous estimate, then the supremum at $t\in [0,T]$ and summing over $q$, we have :
$$\displaylines{\|\begin{pmatrix}
    \tilde{u}_\varepsilon^h \\ \varepsilon \tilde{u}_\varepsilon^v
\end{pmatrix}\|_{\tilde{L}_T^\infty(\mathcal{B}^{0,1/2})}+\sqrt{\nu_h} \|\nabla_h\begin{pmatrix}
    \tilde{u}_\varepsilon^h \\ \varepsilon \tilde{u}_\varepsilon^v
\end{pmatrix}\|_{\tilde{L}_T^2(\mathcal{B}^{0,1/2})} \hfill\cr\hfill +\sqrt{2\varepsilon^{\gamma-2}}\|\partial_z\begin{pmatrix}
    \tilde{u}_\varepsilon^h \\ \varepsilon \tilde{u}_\varepsilon^v
\end{pmatrix}\|_{\tilde{L}_T^2(\mathcal{B}^{0,1/2})}\cr \hfill \leq \|\Delta_q^v \begin{pmatrix}
    \tilde{u}_{\varepsilon,0}^h \\ \varepsilon \tilde{u}_{\varepsilon,0}^v
\end{pmatrix}\|_{L^2}^2+ C \|\begin{pmatrix}
    \tilde{u}_\varepsilon^h \\ \varepsilon \tilde{u}_\varepsilon^v
\end{pmatrix}\|_{\tilde{L}_T^\infty(\mathcal{B}^{0,1/2})}^{1/2} \|\nabla_h \begin{pmatrix}
    \tilde{u}_\varepsilon^h \\ \varepsilon \tilde{u}_\varepsilon^v
\end{pmatrix}\|_{\tilde{L}_T^2(\mathcal{B}^{0,1/2})}\cr\hfill+\frac{C}{\sqrt{\nu_h}}\|\begin{pmatrix}
    \tilde{u}_\varepsilon^h \\ \varepsilon \tilde{u}_\varepsilon^v
\end{pmatrix}\|_{\tilde{L}_T^\infty(\mathcal{B}^{0,1/2})} \|\nabla_h \overline{u}_\varepsilon^h\|_{\tilde{L}_T^2(\mathcal{B}^{0,1/2})}. }$$

Using the same ideas as for the primitive equations from the inequality \eqref{étape de calcul},  we can deduce the estimate \eqref{estimée a priori NS anisotrope} by a standard bootstrap argument.

\subsection{Existence theorem for the system without vertical mean value}
By Friedrichs' method, which we have already used in the proof of the theorem \ref{Caractère bien-posé équations primitives} and which can be found in \cite{BCD}, we deduce Theorem \ref{Caractère bien posé NS}.

\subsection{Uniqueness}
Uniqueness follows from the following stability result:
\begin{lemma}\label{lemme d'unicité NS}~\\
    Let \( u_{\varepsilon,1} \) and \( u_{\varepsilon,2} \) be two divergence-free vector fields such that \( (u_{\varepsilon,1}^h,\varepsilon u_{\varepsilon,1}^v) \) and \( (u_{\varepsilon,2}^h,\varepsilon u_{\varepsilon,2}^v) \) belong to \( \mathcal{C}([0,T];H^{0,1/2}) \), and such that $(\nabla_h (u_{\varepsilon,1}^h,\varepsilon u_{\varepsilon,1}^v)$ and \( \nabla_h (u_{\varepsilon,2}^h,\varepsilon u_{\varepsilon,2}^v) \) belong to \( L_T^2(H^{0,1/2}) \). Let \( w_\varepsilon \) be a divergence-free vector field such that \( (w_{\varepsilon}^h,\varepsilon w_{\varepsilon}^v) \in \mathcal{C}([0,T];H^{0,1/2}) \) and \( \nabla_h (w_{\varepsilon}^h,\varepsilon w_{\varepsilon}^v) \in L_T^2(H^{0,1/2}) \), satisfying the equation:
    $$
    \left\{
    \begin{array}{l}
         \partial_t \begin{pmatrix}
             w_\varepsilon^h \\ \varepsilon w_\varepsilon^v
         \end{pmatrix}
         + u_{\varepsilon,1}\cdot \nabla \begin{pmatrix}
              w_\varepsilon^h \\ \varepsilon w_\varepsilon^v
         \end{pmatrix}
         + w_\varepsilon \cdot \nabla \begin{pmatrix}
              u_{\varepsilon,2}^h \\ \varepsilon u_{\varepsilon,2}^v
         \end{pmatrix}
         - \Delta_h \begin{pmatrix}
             w_\varepsilon^h \\ \varepsilon w_\varepsilon^v
         \end{pmatrix} \\
         \qquad\qquad - \varepsilon^{\gamma-2}\partial_z^2 \begin{pmatrix}
             w_\varepsilon^h \\ \varepsilon w_\varepsilon^v
         \end{pmatrix}
         = -\nabla_\varepsilon p, \\
         \\
         \dive w_\varepsilon = 0.
    \end{array}
    \right.
    $$

Then, for all \( 0 < t < T \), we have the estimate:
\begin{multline}\label{Unicité NS lemme}
\|(w_{\varepsilon}^h,\varepsilon w_{\varepsilon}^v)(t)\|_{H^{0,-1/2}}^2
\leq \|(w_{\varepsilon}^h,\varepsilon w_{\varepsilon}^v)(0)\|_{H^{0,-1/2}}^2 \\ +C \int_0^t f(\tau) \|(w_{\varepsilon}^h,\varepsilon w_{\varepsilon}^v)\|_{H^{0,-1/2}}^2
\times \ln\left(1+e+\frac{1}{\|(w_{\varepsilon}^h,\varepsilon w_{\varepsilon}^v)\|_{H^{0,1/2}}^2}\right) \\
\times \ln\left(\ln\left(1+e+\frac{1}{\|(w_{\varepsilon}^h,\varepsilon w_{\varepsilon}^v)\|_{H^{0,1/2}}^2}\right)\right)d\tau,
\end{multline}
where \( f \) is the time-locally integrable function defined by
\begin{equation}\label{fonction unicité NS}
\begin{aligned}
f(t) = \big(1 + \|(u_{\varepsilon,1}^h,\varepsilon u_{\varepsilon,1}^v)(t)\|_{H^{0,1/2}}^2 + \|(u_{\varepsilon,2}^h,\varepsilon u_{\varepsilon,2}^v)(t)\|_{H^{0,1/2}}^2 \\
+ \|w^h(t)\|_{H^{0,1/2}}^2\big)
\times \big(1 + \|\nabla_h (u_{\varepsilon,1}^h,\varepsilon u_{\varepsilon,1}^v)(t)\|_{H^{0,1/2}}^2 \\
+ \|\nabla_h (u_{\varepsilon,2}^h,\varepsilon u_{\varepsilon,2}^v)(t)\|_{H^{0,1/2}}^2 + \|\nabla_h (w_{\varepsilon}^h,\varepsilon w_{\varepsilon}^v)\|_{H^{0,1/2}}^2\big).
\end{aligned}
\end{equation}
\end{lemma}

\begin{proof}
     By energy method, we have: $$\displaylines{\frac{1}{2}\frac{d}{dt}\|(w_{\varepsilon}^h,\varepsilon w_{\varepsilon}^v)\|_{H^{0,-1/2}}^2+\nu_h\|\nabla_h (w_{\varepsilon}^h,\varepsilon w_{\varepsilon}^v)(t)\|_{H^{0,-1/2}}^2 \hfill\cr +\varepsilon^{\gamma-2}\|\partial_z (w_{\varepsilon}^h,\varepsilon w_{\varepsilon}^v)(t)\|_{H^{0,-1/2}}^2 \cr\hfill \leq \sum_{q\geq -1}2^{-q}\int \Delta_q^v (u_{\varepsilon,1}\cdot \nabla \begin{pmatrix}
              w_\varepsilon^h \\ \varepsilon w_\varepsilon^v
         \end{pmatrix})\Delta_q^v \begin{pmatrix}
              w_\varepsilon^h \\ \varepsilon w_\varepsilon^v
         \end{pmatrix} dx \hfill\cr\hfill +\sum_{q\geq -1}2^{-q}\int \Delta_q^v(w_\varepsilon\cdot\nabla \begin{pmatrix}
              u_{\varepsilon,2}^h \\ \varepsilon u_{\varepsilon,2}^v
         \end{pmatrix})\Delta_q^v\begin{pmatrix}
              w_\varepsilon^h \\ \varepsilon w_\varepsilon^v
         \end{pmatrix} dx.}$$

         Let us estimate the terms on the right-hand side one by one.

         By taking $a=u_{\varepsilon,1}^h$ and $b=\begin{pmatrix}
             w_\varepsilon^h \\ \varepsilon w_\varepsilon^v
         \end{pmatrix}$ in Lemma \ref{Premier sous-lemme d'unicité}, we get: 
         \begin{multline}\label{Unicité NS 1}
             \sum_{q\geq -1}2^{-q}\int \Delta_q^v(u_{\varepsilon,1}^h\cdot\nabla_h \begin{pmatrix}
             w_\varepsilon^h \\ \varepsilon w_\varepsilon^v
         \end{pmatrix})\Delta_q^v \begin{pmatrix}
             w_\varepsilon^h \\ \varepsilon w_\varepsilon^v
         \end{pmatrix} \ dx \\ \leq Cf(t)(1-\ln \|(w_\varepsilon^h, \varepsilon w_\varepsilon^v)\|_{H^{0,-1/2}}^2)(1+\ln(\|(w_\varepsilon^h, \varepsilon w_\varepsilon^v)\|_{H^{0,-1/2}}^2))\hfill\cr\hfill+\frac{\nu_h}{20}\|\nabla_h (w_\varepsilon^h, \varepsilon w_\varepsilon^v)\|_{H^{0,-1/2}}^2.
         \end{multline}

         By Lemma \ref{Deuxième sous-lemme d'unicité} with $a=u_{\varepsilon,2}^h$ and $b=w_\varepsilon^h$, we have : 
$$\displaylines{\sum_{q\geq -1}2^{-q}\int \Delta_q^v(w_\varepsilon^h\cdot\nabla_h u_{\varepsilon,2}^h)\Delta_q^v w_\varepsilon^h dx\leq C f(t)(1-\ln \|w_\varepsilon^h\|_{H^{0,-1/2}}^2)\|w_\varepsilon^h\|_{H^{0,-1/2}}^2 \hfill\cr\hfill+\frac{\nu_h}{20}\|\nabla_h w_\varepsilon^h\|_{H^{0,-1/2}}^2.}$$
         By proceeding analogously to the proof of the lemma \ref{Deuxième sous-lemme d'unicité} for $$\sum_{q\geq -1}2^{-q}\int \Delta_q^v(w_\varepsilon^h\cdot\nabla_h \varepsilon u_{\varepsilon,2})\Delta_q^v (\varepsilon w_\varepsilon^v)dx,$$ we have the following inequality :
         \begin{multline}\label{Unicité NS 2}
             \sum_{q\geq -1}2^{-q}\int \Delta_q^v(w_\varepsilon\cdot\nabla \begin{pmatrix} u_{\varepsilon,2}^h \\ \varepsilon u_{\varepsilon,2}^v \end{pmatrix})\Delta_q^v \begin{pmatrix} w_\varepsilon^h \\ \varepsilon w_\varepsilon^v \end{pmatrix} dx
             \\ \leq C f(t)(1-\ln \|(w_\varepsilon^h, \varepsilon w_\varepsilon^v)\|_{H^{0,-1/2}}^2)\|(w_\varepsilon^h, \varepsilon w_\varepsilon^v)\|_{H^{0,-1/2}}^2 \\ +\frac{\nu_h}{20}\|\nabla_h (w_\varepsilon^h, \varepsilon w_\varepsilon^v)\|_{H^{0,-1/2}}^2.
         \end{multline}

    Taking $a=u_{\varepsilon,1}$ and $b=\begin{pmatrix}
        w_{\varepsilon}^h \\ \varepsilon w_{\varepsilon}^v
    \end{pmatrix}$ in Lemma \ref{Troisième sous-lemme d'unicité} leads to:
    \begin{multline}\label{Unicité NS 3}
        \sum_{q\geq -1}2^{-q}\int \Delta_q^v(u_{\varepsilon,1}^v\partial_z \begin{pmatrix}
        w_{\varepsilon}^h \\ \varepsilon w_{\varepsilon}^v
    \end{pmatrix})\Delta_q^v \begin{pmatrix}
        w_{\varepsilon}^h \\ \varepsilon w_{\varepsilon}^v
    \end{pmatrix} dx \\ \leq Cf(t)(1-\ln \|(w_{\varepsilon}^h,\varepsilon w_{\varepsilon}^v)\|_{H^{0,-1/2}}^2)+\frac{2\nu_h}{25}\|\nabla_h (w_{\varepsilon}^h,\varepsilon w_{\varepsilon}^v)\|_{H^{0,-1/2}}^2.
    \end{multline}

    By Lemma \ref{Quatrième sous-lemme d'unicité}, with $a=u_{\varepsilon,2}^h$ and $b=w_\varepsilon$, we get: $$\displaylines{\sum_{q\geq -1}2^{-q}\int \Delta_q^v(w_\varepsilon^v \partial_z u_{\varepsilon,2}^h)\Delta_q^v w_\varepsilon^h dx \hfill\cr\hfill \leq Cf(t)(1-\ln \|w_\varepsilon^h(t)\|_{H^{0,-1/2}}^2)\ln(1-\ln \|w_\varepsilon^h(t)\|_{H^{0,-1/2}}^2)\|w_\varepsilon^h\|_{H^{0,-1/2}}^2 \hfill\cr\hfill +\frac{3\nu_h}{25}\|\nabla_h w_\varepsilon^h\|_{H^{0,-1/2}}^2.}$$ 

    For the term $$\sum_{q\geq -1}2^{-q}\int \Delta_q^v(w_\varepsilon^v \partial_z \varepsilon u_{\varepsilon,2}^v)\Delta_q^v \varepsilon w_\varepsilon^v dx, $$ we proceed in a similar way to the proof of Lemma \ref{Quatrième sous-lemme d'unicité} (here $w_\varepsilon^v$ is not a field of dimension 2 but a scalar), and we obtain: 
    \begin{multline}\label{Unicité NS 4}
        \sum_{q\geq -1}2^{-q}\int \Delta_q^v(w_\varepsilon^v \partial_z \begin{pmatrix} u_{\varepsilon,2}^h \\ \varepsilon u_{\varepsilon,2}^v\end{pmatrix})\Delta_q^v \begin{pmatrix} w_\varepsilon^h \\ \varepsilon w_\varepsilon^v \end{pmatrix} dx \hfill\cr\hfill \leq Cf(t)(1-\ln \|(w_\varepsilon^h,\varepsilon w_\varepsilon^v)(t)\|_{H^{0,-1/2}}^2)\ln(1-\ln \|(w_\varepsilon^h,\varepsilon w_\varepsilon^v)(t)\|_{H^{0,-1/2}}^2)\hfill\cr\hfill\times\|(w_\varepsilon^h,\varepsilon w_\varepsilon^v)\|_{H^{0,-1/2}}^2 +\frac{3\nu_h}{25}\|\nabla_h(w_\varepsilon^h,\varepsilon w_\varepsilon^v)\|_{H^{0,-1/2}}^2.
    \end{multline}

    By summing up the inequalities \eqref{Unicité NS 1}, \eqref{Unicité NS 2}, \eqref{Unicité NS 3} and \eqref{Unicité NS 4}, we obtain the inequality \eqref{Unicité NS lemme}.
\end{proof}

Let $u_{\varepsilon,1}$ and $u_{\varepsilon,2}$ be two divergence-free vector fields with the same initial condition such that $(u_{\varepsilon,1}^h,\varepsilon u_{\varepsilon,1}^v)$ and $(u_{\varepsilon,2}^h,\varepsilon u_{\varepsilon,2}^v)$ belong to the space $L_T^\infty(H^{0,1/2})$ and such that $\nabla_h (u_{\varepsilon,1}^h,\varepsilon u_{\varepsilon,1}^v)$ and $\nabla_h (u_{\varepsilon,2}^h,\varepsilon u_{\varepsilon,2}^v)$ belong to the space $L_T^2(H^{0,1/2})$. By setting $w_\varepsilon=u_{\varepsilon,2}-u_{\varepsilon,1}$, we have that $u_{\varepsilon,1}$, $u_{\varepsilon,2}$ and $w_\varepsilon$ verify the assumptions of Lemma \ref{lemme d'unicité NS}.

By definition of $w_\varepsilon$, we have $$\|(w_{\varepsilon}^h,\varepsilon w_\varepsilon^v)\|_{L_T^\infty(H^{0,1/2})}\leq \|(u_{\varepsilon,1}^h,\varepsilon u_{\varepsilon,1}^v)\|_{L_T^\infty(H^{0,1/2})}+\|(u_{\varepsilon,2}^h,\varepsilon u_{\varepsilon,2}^v)\|_{L_T^\infty(H^{0,1/2})},$$ and also $$\displaylines{\|\nabla_h (w_{\varepsilon}^h,\varepsilon w_{\varepsilon}^v)\|_{L_T^2(H^{0,1/2})}\leq \|\nabla_h (u_{\varepsilon,1}^h,\varepsilon u_{\varepsilon,1}^v)\|_{L_T^2(H^{0,1/2})} \hfill\cr\hfill+\|\nabla_h (u_{\varepsilon,2}^h,\varepsilon u_{\varepsilon,2}^v)\|_{L_T^2(H^{0,1/2})}.}$$ We therefore obtain that $f$ (defined in \eqref{fonction unicité NS}) is a locally integrable function.
In addition, we can see that we have the following continuity relationship between the solution and the initial data:
$$\displaylines{\|(u_{\varepsilon,1}^h,\varepsilon u_{\varepsilon,1})(t)-(u_{\varepsilon,2}^h,\varepsilon u_{\varepsilon,2})(t)\|_{H^{0,-1/2}}^2 \hfill\cr\hfill \leq C \exp\left(-(-\ln \|(u_{\varepsilon,1}^h,\varepsilon u_{\varepsilon,1})(0)-(u_{\varepsilon,2}^h,\varepsilon u_{\varepsilon,2})(0)\|_{H^{0,-1/2}}^2)^{\exp(\int_0^t f(s)ds)}\right).}$$ Osgood Lemma \ref{Osgood} with $\rho(s)\mathrel{\mathop:}=\|(w_{\varepsilon}^h,\varepsilon w_{\varepsilon})(s)\|_{H^{0,-1/2}}^2$ and $\mu(r)\mathrel{\mathop:}=r(1-\ln r)\ln(1-\ln r)$, implies uniqueness. 

\section{Convergence}
We now prove Theorem \ref{théorème de convergence}.

From the a priori estimates \eqref{estimée a priori NS anisotrope}, we in particular have:
\begin{itemize}
    \item[$\bullet$] $(u_\varepsilon^h,\varepsilon u_\varepsilon^v)_\varepsilon$ is a bounded sequence in $L^\infty(\R_+;\mathcal{B}^{0,1/2})$, hence also in $L^\infty(\R_+;L^2)$.
    \item[$\bullet$] $(\nabla_h(u_\varepsilon^h,\varepsilon u_\varepsilon^v))_\varepsilon$ is a bounded sequence in $L^2(\R_+;\mathcal{B}^{0,1/2})$, so $(u_\varepsilon^h, \varepsilon u_\varepsilon^v)_\varepsilon$ is bounded in $L^2(\R_+; H^\eta)$ for any $\eta<1/2$.
\end{itemize}

We also have that $(\partial_t(u_\varepsilon^h, \varepsilon u_\varepsilon^v))_\varepsilon$ is a bounded sequence in the space $L^2(\R_+; H^{-3})$, using the rescaled equation \eqref{NS remise à l'échelle}. Indeed,
\begin{itemize}
    \item[$\bullet$] $-\nu_h \Delta_h \begin{pmatrix}
        u_\varepsilon^h \\ \varepsilon u_\varepsilon^v
    \end{pmatrix}-\varepsilon^{\gamma-2} \partial_z^2 \begin{pmatrix}
         u_\varepsilon^h \\ \varepsilon u_\varepsilon^v
    \end{pmatrix}$ belongs to $L^2(\R_+;H^{-2})$ uniformly in $\varepsilon$ (using also the information that the sequence
    $$\left(-\varepsilon^{\gamma-2} \partial_z^2 \begin{pmatrix}
         u_\varepsilon^h \\ \varepsilon u_\varepsilon^v
    \end{pmatrix}\right)_\varepsilon$$
    is bounded in $L^2(\R_+;\mathcal{B}^{0,1/2})$, provided by the a priori estimates \eqref{estimée a priori NS anisotrope});
    
    \item[$\bullet$] Since we have $u_\varepsilon \cdot\nabla \begin{pmatrix}
        u_\varepsilon^h \\ \varepsilon u_\varepsilon^v
    \end{pmatrix}=\dive(u_\varepsilon\otimes \begin{pmatrix}
        u_\varepsilon^h \\ \varepsilon u_\varepsilon^v
    \end{pmatrix})$ and since $L^1(\Omega)$ continuously embeds into $H^{-s}(\Omega)$ for $s>3/2$, we obtain:
    \begin{align*}
        \|u_\varepsilon\cdot \nabla \begin{pmatrix}
        u_\varepsilon^h \\ \varepsilon u_\varepsilon^v
    \end{pmatrix}\|_{L_t^2(H_x^{-3})}  & \leq \|u_\varepsilon\otimes \begin{pmatrix}
        u_\varepsilon^h \\ \varepsilon u_\varepsilon^v
    \end{pmatrix}\|_{L_t^2(H_x^{-2})} \\ & \leq C \|u_\varepsilon\otimes \begin{pmatrix}
        u_\varepsilon^h \\ \varepsilon u_\varepsilon^v
    \end{pmatrix}\|_{L_t^2(L_x^1)} \\ & \leq C \|u_\varepsilon\|_{L_t^2(L_x^2)}\|(u_\varepsilon^h,\varepsilon u_\varepsilon^v)\|_{L_t^\infty(L_x^2)}.
    \end{align*}
    Since $(u_\varepsilon^h,\varepsilon u_\varepsilon^v)$ is bounded in
    $$L^\infty(\R_+;L^2)\cap L^2(\R_+; H^{-\eta}) \quad \text{with} \ \eta<1/2,$$
    and using the inequality
    $$\|u_\varepsilon^v\|_{L_t^2(L_x^2)}\leq \|\partial_z u_\varepsilon^v\|_{L_t^2(L_x^2)}\leq \|\nabla_h u_\varepsilon^h\|_{L_t^2(L_x^2)}$$
    (from the vertical Poincaré inequality \ref{Poincaré vertical} and the divergence-free condition),
    we conclude that $\left(u_\varepsilon\cdot \begin{pmatrix}
        u_\varepsilon^h \\ \varepsilon u_\varepsilon^v
    \end{pmatrix}\right)_\varepsilon$ is a bounded sequence in $L^2(\R_+;H^{-3})$.
    
    \item[$\bullet$] We have the identity $\nabla_\varepsilon p_\varepsilon=\nabla_\varepsilon(-\Delta)^{-1}\dive_\varepsilon(u_\varepsilon\cdot\nabla)\begin{pmatrix}
        u_\varepsilon^h \\ \varepsilon u_\varepsilon^v
    \end{pmatrix}$, which shows (by the same arguments as above) that $(\nabla_\varepsilon p_\varepsilon)_\varepsilon$ is a bounded sequence in $L^2(\R_+; H^{-3})$.
\end{itemize}

Therefore, by the Aubin–Lions lemma, $(u_\varepsilon^h)$ converges (up to a subsequence) strongly to some $u^h$ in $\mathcal{C}_{loc}(\R_+;H_{loc}^{-3})$.

Moreover, $(\varepsilon u_\varepsilon^v)_\varepsilon$ converges strongly to $0$ in $L^2(\R_+;L^2)$ since, using the Poincaré inequality \ref{Poincaré vertical} and the divergence-free condition, we have
$$\|\varepsilon u_\varepsilon^v\|_{L_t^2(L_x^2)}\leq \varepsilon \|\partial_z u_\varepsilon^v\|_{L_t^2(L_x^2)} \leq  \varepsilon \|\nabla_h u_\varepsilon^h\|_{L_t^2(L_x^2)}\underset{\varepsilon\to 0}{\longrightarrow}0.$$

Because $\left(\begin{pmatrix}
    u_\varepsilon^h \\ \varepsilon u_\varepsilon^v
\end{pmatrix}\right)_\varepsilon$ is bounded in $L^2(H^\eta)$ for all $\eta<1/2$, interpolation yields that $\begin{pmatrix}
    u_\varepsilon^h \\ \varepsilon u_\varepsilon^v
\end{pmatrix}$ converges to $\begin{pmatrix}
    u^h \\ 0
\end{pmatrix}$ in $L_{loc}^2(H_{loc}^\eta)$.

Additionally, we know that the sequence $(u_\varepsilon^v)_\varepsilon$ converges weakly to some $u^v$ in $L^2(\R_+; L^2)$ due to the inequality $\|u_\varepsilon^v\|_{L_t^2(L_x^2)}\leq \|\nabla_h u_\varepsilon^h\|_{L_t^2(L_x^2)}$.

In particular, we have in the sense of distributions:
$$u_\varepsilon\cdot\nabla \begin{pmatrix}
    u_\varepsilon^h \\ \varepsilon u_\varepsilon^v
\end{pmatrix}=\dive\left(u_\varepsilon\otimes \begin{pmatrix}
    u_\varepsilon^h \\ \varepsilon u_\varepsilon^v
\end{pmatrix}\right)\underset{\varepsilon\to 0}{\longrightarrow}\dive\left(u\otimes\begin{pmatrix}
    u^h \\ 0
\end{pmatrix} \right)=u\cdot \nabla\begin{pmatrix}
    u^h \\ 0
\end{pmatrix}.$$
\\
Indeed, $(u_\varepsilon)_\varepsilon$ converges weakly to $u$ in $L_{loc}^2(\R_+;L^2)$, $(u_\varepsilon^h)_\varepsilon$ converges strongly to $u^h$ in $L_{loc}^2(H_{loc}^\eta)$, and $(\varepsilon u_\varepsilon^v)_\varepsilon$ converges strongly to $0$ in $L^2(\R_+;L^2)$.

We thus conclude that the limit $u$ (which is unique due to the uniqueness of solutions to the primitive equations) satisfies the primitive equations \eqref{Equations primitives}, which proves the result.

\appendix
\section{Anisotropic Littlewood-Paley theory}\label{Théorie de Littlewood-Paley anisotrope}
Here we recall some elements of the anisotropic Littlewood-Paley theory from \cite{Paicu1}.

First, recall the definition of anisotropic Lebesgue spaces. Let $L_h^p(L_v^q)$ denote the space of functions defined on $\Omega$ that belong to $L^p(\Omega_h;L^q([-1,1])$, with the norm $$\|f\|_{L_h^p(L_v^q)}\mathrel{\mathop:}=\| \ \|f(x_h,\cdot)\|_{L^q([-1,1])}\|_{L^p(\Omega_h)} \quad \text{for } \ 1 \leq p,q\leq \infty .$$

Similarly, we denote $L_v^q(L_h^p)$ the space $L^q([-1,1]; L^p(\Omega_h))$ with the associated norm $\|f\|_{L_v^q(L_h^p)}\mathrel{\mathop:}=\| \ \|f(\cdot,x_v)\|_{L^p(\Omega_h)}\|_{L^q([-1,1])}$. 

The following lemma shows the importance of the order of integration:
\begin{lemma}\label{ordre intégration}
    Let be $1 \leq p\leq q\leq \infty$ and $f\in L^p(X_1;L^q(X_2))$ where $(X_1,d\mu_1)$ and $(X_2,d\mu_2)$ are measurable spaces. Then $f\in L^q(X_2;L^p(X_1))$ with continuous embedding: $$\|f\|_{L^q(X_2;L^p(X_1))}\leq \|f\|_{L^p(X_1;L^q(X_2))}.$$ 
\end{lemma}

Within the framework of anisotropic Littlewood-Paley theory, the truncation operators with respect to the vertical variable are defined as follows $$\begin{array}{ll}
     \displaystyle\Delta_q^v u=\sum_{n\in\Z^3}\mathcal{F}u_n \phi(\frac{|\check{n}_3|}{2^q})e^{i\pi \check{n}\cdot x}; & \text{for} \ q\geq 0 \\
     \displaystyle\Delta_{-1}^v u=\sum_{n\in\Z^3} \mathcal{F }u_n \chi(|\check{n}_3|)e^{i\pi \check{n}\cdot x} & \\
      \displaystyle\Delta_q^v u=0; & \text{for} \ q\geq -2 
\end{array}$$
where $u\in \mathcal{D}'(\Omega_1)$, $\check{n}=(\frac{n_1}{2},\frac{n_2}{2},\frac{n_3}{2})$. The functions $\phi$ and $\chi$ generate a dyadic partition of the unit in $\R$: they are regular, satisfy $$\supp \ \chi\subset [-4/3;4/3], \quad \supp \ \phi \subset \mathcal{C}(3/4,8/3)\mathrel{\mathop:}=\{\xi\in\R,\:  3/4\leq|\xi|\leq 8/3\},$$
and $$\chi(t)+\sum_{q\geq 0}\phi(2^{-q}t)=1.$$

We can generalise this definition to the case of $\Omega_2$. In this context, we define : $$\Delta_q^v u(x^h,x^v)=\sum_{n\in\Z}\mathcal{F}_{h}^{-1}(\varphi(\cdot,2^{-q}n)\mathcal{F}_h u)(x^h,x^v)\times \widehat{u}_n e^{i\pi n x^v}, \quad \text{for} \ q\geq 0$$ where  $\displaystyle\widehat{u}_n=\frac{1}{2} \int_{-1}^1 e^{-i \pi nz}u(x^h,z)dz$ and $\mathcal{F}_h$ is the Fourier transform in the horizontal variable $x^h$.

Let us also define the operator: $$S_q^v a\mathrel{\mathop:}=\sum_{q'\leq q-1}\dot\Delta_{q'}^v a.$$

The following lemma shows Bernstein-type inequalities and Gagliardo-Nirenberg-type estimates:

\begin{lemma}\label{Bernstein et Gagliardo-Nirenberg}
    Let $u$ be a function with $\supp \mathcal{F}^v u\subset2^q\mathcal{C}$, where we have noted $\mathcal{F}^v$ the Fourier transform in the vertical variable and where $\mathcal{C}=[-b;-a]\cup[a;b]$ with $0<a<b$. Let $p\geq 1$ and $r\geq r'\geq 1$ be two real numbers. We have : $$\begin{array}{l}
         2^{qk}C^{-k}\|u\|_{L_h^p(L_v^r)}\leq \|\partial_v^k u\|_{L_h^p(L_v^r)}\leq 2^{qk}C^k\|u\|_{L_h^p(L_v^r)}; \\
          2^{qk}C^{-k}\|u\|_{L_v^r(L_h^p)}\leq \|\partial_v^k u\|_{L_v^r(L_h^p)}\leq 2^{qk}C^k\|u\|_{L_v^r(L_h^p)};
    \end{array}$$
    $$\begin{array}{l}
         \|u\|_{L_h^p(L_v^r)}\leq C2^{q(\frac{1}{r'}-\frac{1}{r})}\|u\|_{L_h^p(L_v^{r'})};  \\
          \|u\|_{L_v^r(L_h^p)}\leq C2^{q(\frac{1}{r'}-\frac{1}{r})}\|u\|_{L_v^{r'}(L_h^p)}. 
    \end{array}$$

    For the case of $\Omega_2$, we have the "classical" Gagliardo-Nirenberg inequality:
\begin{equation}\label{classique GN}\|u\|_{L^2([-1,1];L^4(\Omega_h))}\leq \|u\|_{L^2(\Omega)}^{1/2}\|\nabla_h u\|_{L^2(\Omega)}^{1/2}.\end{equation}
    
    For the case of the three-dimensional torus $\Omega_1$, the previous inequality  \eqref{classique GN} remains true assuming $u$ of zero horizontal mean value.

More generally, we have the following inequality : \begin{equation}\label{inégalité de Gagliardo-Nirenberg généralisée}
    \|u\|_{L^2([-1,1];L^4(\Omega_h))}\leq (\|u\|_{L^2(\Omega)}^{1/2}\|\nabla_h u\|_{L^2(\Omega)}^{1/2}+\|u\|_{L^2(\Omega)}).
\end{equation}
    
    Moreover, if $u$ is such that the support of $\mathcal{F}^v u$ is included in $2^q \mathcal{C}$, then we have :
    $$\|u\|_{L^4(\Omega_h;L^\infty([-1,1]))}\leq 2^{q/2}\|u\|_{L^2(\Omega)}^{1/2}\|\nabla_h u\|_{L^2(\Omega)}^{1/2}.$$
\end{lemma}

\begin{coro}\label{coro Bernstein-Gagliardo-Nirenberg}
    We have the inequalities: $$\begin{array}{l}
         \|u\|_{L_h^2(L_v^\infty)}\leq C\|u\|_{\mathcal{B}^{0,1/2}} \\ \|u\|_{L_h^4(L_v^\infty)}\leq C \|u\|_{\mathcal{B}^{0,1/2}}^{1/2} \|\nabla_h u\|_{\mathcal{B}^{0,1/2}}^{1/2}.
          
    \end{array}$$
    For periodic functions, in the last inequality, $u$ is assumed to have zero horizontal mean value.
\end{coro}

This appendix ends with a reminder of the paradifferential anisotropic calculus.
We note $$T_u w\mathrel{\mathop:}=\sum_{q'\leq q-2}\Delta_{q'}^v u\dot \Delta_q^v w=\sum_q S_{q-1}^v u\cdot \Delta_q^v w$$ and $$R(u,w)\mathrel{\mathop:}=\sum_q \sum_{i\in\{0,\pm 1\}}\Delta_q^v u \Delta_{q+i}^v w.$$

We then have $$u\cdot w=T_u w+T_w u+R(u,w). $$ Given that the support of $\mathcal{F}(S_{q-1}^w u\Delta_q^v w)$ is included in a certain ring $2^q \mathcal{C}$ and taking into account the fact that the support of $\mathcal{F}(S_{q+1}^v w\Delta_q^v u)$ is included in an interval $2^q B$ where $B=[-c, c]$ with $c>0$, we deduce the following quasi-orthogonality properties
\begin{equation}\label{Quasi-orthogonalité}\begin{array}{l}
     \Delta_q^v (S_{q'-1}^v u \Delta_{q'}^v w)=0 \ \text{for} \ |q'-q|\geq 5 \ \text{and also} \\
      \Delta_q^v (S_{q'+1}^v u \Delta_{q'}^v w)=0 \ \text{for} \ q'\leq q-4.
\end{array}\end{equation}

\begin{lemma}\label{estimée de commutateur}
    For $p\geq 1$, $r\geq 1$, $s\geq 1$ with $\frac{1}{p}=\frac{1}{r}+\frac{1}{s}$, we have the following inequality : 
    $$\|[\Delta_q^v, u]w\|_{L_v^2(L_h^p)}\leq C 2^{-q}\|\partial_z u\|_{L_v^\infty(L_h^r)}\|w\|_{L_v^2(L_h^s)}.$$
\end{lemma}

We recall here a Sobolev injection lemma whose proof in the whole space comes from \cite{BCD} :
\begin{lemma}
    Let $s\in ]0,\frac{d}{2}[$. There is a constant $C_d$ depending on the dimension such that :
    \begin{equation}\label{injection de Sobolev espace entier}\|u\|_{L^p(\R^d)}\leq C_d \sqrt{p} \|u\|_{\dot H^s(\R^d)} \quad \text{with} \  p=\frac{2d}{d-2s}.\end{equation}

    In the case of the torus, we have:
    \begin{equation}\label{injection de Sobolev tore}\|u\|_{L^p(\mathcal{T}^d)}\leq C_d \sqrt{p} \|u\|_{H^s(\mathcal{T}^d)} \quad \text{with} \  p=\frac{2d}{d-2s}.\end{equation}
\end{lemma}

The following lemma has been useful for uniqueness:
\begin{lemma}\label{lemme Sobolev}
    There is a constant $C$, such that we have the following inequality for $u$ with zero horizontal mean value
    \begin{equation}\label{Inégalité de Gagliardo-Nirenberg sans moyenne}\|u\|_{L^{2p}(\Omega_h)}\leq C\sqrt{p}\|u\|_{L^2(\Omega_h)}^{1/p}\|\nabla_h u\|_{L^2(\Omega_h)}^{1-1/p} \quad \text{for all} \ 1\leq p<+\infty.\end{equation}

    More generally, we have the following inequality for the torus case:
    \begin{equation}\label{Inégalité Gagliardo-Nirenberg avec mooyenne}\|u\|_{L^{2p}(\Omega_h)}\leq C\sqrt{p}(\|u\|_{L^2(\Omega_h)}^{1/p}\|\nabla_h u\|_{L^2(\Omega_h)}^{1-1/p}+\|u\|_{L^2(\Omega_h)})\end{equation} \text{for all}  $1\leq p<+\infty$.
\end{lemma}

\section{}

Here we recall two Poincaré-Wirtinger inequalities and Osgood's lemma.
\begin{lemma}\label{Poincaré horizontal} Let $f\in \mathcal{C}_0^\infty(\mathcal{T}^3)$, we have the existence of a constant $C$ such that : $$\|f(x_h,\cdot)-\frac{1}{4}\int_{\mathcal{T}_h^2} f(x_h',\cdot) \ dx_h'\|_{L^p(\mathcal{T}^3)}\leq C \|\nabla_h f\|_{L^p(\mathcal{T}^3)}.$$
    Furthermore, if $f$ has zero horizontal mean value, then we have $$\|f\|_{L^p(\mathcal{T}^3)}\leq C \|\nabla_h f\|_{L^p(\mathcal{T}^3)}.$$ 
\end{lemma}

\begin{lemma}\label{Poincaré vertical} Let $f\in \mathcal{C}_0^\infty(\Omega)$, we have : $$\|f(\cdot,z)-\frac{1}{2}\int_{-1}^1 f(\cdot, z') \ dz'\|_{L^p(\Omega)}\leq 2 \|\partial_z f\|_{L^p(\Omega)}.$$
    Moreover, if $f$ is odd with respect to the vertical variable, then we have $$\|f\|_{L^p(\Omega)}\leq 2 \|\partial_z f\|_{L^p(\Omega)}.$$ 
\end{lemma}

\begin{lemma}\label{Osgood} Let $\rho\geq 0$ be a measurable function, $\gamma\geq 0$ a locally integrable function and $\mu\geq 0$ an increasing continuous function that satisfies the condition
    $$\int_0^1 \frac{dr}{\mu(r)}=+\infty.$$

    Let $a$ also be a positive real number. If $\rho$ satisfies the inequality $$\rho
    (t)\leq a+\int_0^t \gamma(s)\mu(\rho(s))ds,$$

    then, either $a$ is zero and the $\rho$ function is identically zero or $a$ is non-zero and we have
    $$-\mathcal{M}(\rho(t))+\mathcal{M}(a)\leq \int_0^t \gamma(s) ds, \quad \text{with} \quad \mathcal{M}(x)=\int_x^1 \frac{dr}{\mu(r)}.$$
\end{lemma}

\section{Convection lemmas}\label{convection}
In order to prove the a priori estimates of the system  \eqref{Equations primitives sans moyenne verticale},  we will need the following first convection lemma:

\begin{lemma}\label{Premier lemme de convection}
    For the case $\Omega=\Omega_2$, there exists a constant $C$ such that for any time-dependent vector field with zero divergence $a(t)$ and scalar function $b(t)$ with $a^h$ and $b$ belonging to
    \begin{equation}\label{Espace fonctionnel E}
        E\mathrel{\mathop:}=\left\{u\in \tilde{L}_T^\infty(\mathcal{B}^{0,1/2}) \ \middle| \ \nabla_h u\in \tilde{L}_T^2(\mathcal{B}^{0,1/2}) \right\},
    \end{equation}
     we have the existence of a sequence $(c_q)_{q\geq -1}$ such that $\sum_{q\geq -1}\sqrt{c_q}\leq 1$ and $$\displaylines{\int_0^T \left|(\Delta_q^v(a\cdot\nabla b)|\Delta_q^v b)_{L^2(\Omega)}\right|dt \hfill\cr\hfill \leq C c_q 2^{-q}\bigg(\|a^h\|_{\tilde{L}_T^\infty(\mathcal{B}^{0,1/2})}^{1/2}\|\nabla_h a^h\|_{\tilde{L}_T^2(\mathcal{B}^{0,1/2})}^{1/2}\|b\|_{\tilde{L}_T^\infty(\mathcal{B}^{0,1/2})}^{1/2}\|\nabla_h b\|_{\tilde{L}_T^2(\mathcal{B}^{0,1/2})}^{3/2} \hfill\cr\hfill+\|\nabla_h a^h\|_{\tilde{L}_T^2(\mathcal{B}^{0,1/2})}\|b\|_{\tilde{L}_T^\infty(\mathcal{B}^{0,1/2})}\|\nabla_h b\|_{\tilde{L}_T^2(\mathcal{B}^{0,1/2})}\bigg).}$$ 

    For the case $\Omega=\Omega_1$, we have the same result, adding the hypothesis that $a$ and $b$ both have zero mean values on $\Omega_1$.
\end{lemma}
\begin{remark}~~
In the case of the three-dimensional torus, the assumption $a$ and $b$ of zero horizontal mean values allows us to use the Gagliardo-Nirenberg inequality \eqref{classique GN} of Lemma \ref{Bernstein et Gagliardo-Nirenberg} which is the same as that for the case of whole space.

In the following lemmas, we will consider the horizontal mean values of $a$ and $b$ for the three-dimensional torus.
\end{remark}

\begin{proof}
First of all, let us split the left-hand side of this inequality into two parts: a convection term containing only horizontal derivatives and another containing the vertical derivative

    $$I\mathrel{\mathop:}=\int_0^T|(\Delta_q^v(a\cdot\nabla b)|\Delta_q^v b)_{L^2(\Omega)}|dt=I_h+I_v,$$
    where
    \begin{align*}
        I_h & \mathrel{\mathop:}= \int_0^T|(\Delta_q^v(a^h\cdot\nabla_h b)|\Delta_q^v b)_{L^2(\Omega)}|dt
        \\
        \text{and} \quad I_v & \mathrel{\mathop:}= \int_0^T|(\Delta_q^v(a^v\partial_z b)|\Delta_q^v b)_{L^2(\Omega)}|dt.
    \end{align*}

\subsection*{Horizontal term} Let us first estimate the term $I_h$ and let us denote $F_q^h=\Delta_q^v(a^h\cdot \nabla_h b)$. 
    The idea is to look at $\Delta_q^v b$ in $L_h^4(L_v^2)$ in order to use the Gagliardo-Nirenberg inequality \eqref{classique GN} and the fact that $b\in E$ (defined in \eqref{Espace fonctionnel E}). We will decompose using the paradifferential calculus $F_q^h$ and estimate this term in the "classical" way.
    
    By the Cauchy-Schwarz inequality in the vertical variable, we have:
    $$|(F_q^h|\Delta_q^v b)_{L^2(\Omega)}|\leq \int \|F_q^h\|_{L_v^2}\|\Delta_q^v b\|_{L_v^2}dx_h.$$

    Using Hölder's inequality in the horizontal variable, we obtain :  $$|(F_q^h|\Delta_q^v b)_{L^2(\Omega)}|\leq \|F_q^h\|_{L_h^{4/3}(L_v^2)}\|\Delta_q^v b\|_{L_h^4(L_v^2)}.$$

    By Lemma \ref{ordre intégration} and the Gagliardo-Nirenberg inequality \eqref{classique GN}, we get : 
    $$|(F_q^h|\Delta_q^v b)_{L^2(\Omega)}|\leq \|F_q^h\|_{L_h^{4/3}(L_v^2)}\|\Delta_q^v b\|_{L^2(\Omega)}^{1/2}\|\Delta_q^v \nabla_h b\|_{L^2(\Omega)}^{1/2}.$$

    We then have by Hölder's inequality in time :
    $$\displaylines{\int_0^T |(F_q^h|\Delta_q^v b)_{L^2(\Omega)}|d\tau\leq \|F_q^h\|_{L_T^{4/3}(L_h^{4/3}L_v^2)}\|\Delta_q^v b\|_{L_T^\infty(L^2(\Omega))}^{1/2}  \hfill\cr\hfill\times\left\| \|\Delta_q^v \nabla_h b\|_{L^2(\Omega)}^{1/2} \right\|_{L^4([0,T])}.}$$

    So we have : $$\displaylines{\int_0^T |(F_q^h|\Delta_q^v b)_{L^2(\Omega)}|d\tau\leq \|F_q^h\|_{L_T^{4/3}(L_h^{4/3}L_v^2)}\|\Delta_q^v b\|_{L_T^\infty(L^2(\Omega))}^{1/2}  \hfill\cr\hfill\times \|\Delta_q^v \nabla_h b\|_{L_T^2(L^2(\Omega))}^{1/2}.}$$

     Since $b$ belongs to $E$, if we note $c_q^{(1)}=\frac{2^{q/2}\|\Delta_q^v b\|_{L_T^\infty(L^2(\Omega))}}{\|b\|_{\tilde{L}_T^\infty(\mathcal{B}^{0,1/2})}}\in l_q^1$, $c_q^{(2)}=\frac{2^{q/2}\|\Delta_q^v \nabla_h b\|_{L_T^2(L^2(\Omega))}}{\|\nabla_h b\|_{\tilde{L}_T^2(\mathcal{B}^{0,1/2})}}\in l_q^1$ and $d_q=\sqrt{c_q^{(1)}c_q^{(2)}}$, then $$\displaylines{\int_0^T |(F_q^h(\tau)|\Delta_q^v b(\tau))_{L^2}|d\tau\leq 2^{-q/2}C d_q \|F_q^h\|_{L_T^{4/3}(L_h^{4/3}L_v^2)}\|b\|_{\tilde{L}_T^\infty(\mathcal{B}^{0,1/2})}^{1/2} \hfill\cr\hfill \|\nabla_h b\|_{\tilde{L}_T^2(\mathcal{B}^{0,1/2})}^{1/2}.}$$ 

Now let us estimate $\|\Delta_q^v(a^h\cdot\nabla_h b)\|_{L_h^{4/3}(L_v^2)}$. Decomposing in terms of paraproducts and remainder, we have $$\displaylines{\Delta_q^v(a^h\cdot\nabla_h b)=\Delta_q^v T_{a^h}\nabla_h b+\Delta_q^v T_{\nabla_h b}a^h+\Delta_q^v R(a^h,\nabla_h b). }$$

Let us estimate these terms one by one.
\\
$\bullet$ For the first term $\Delta_q^v T_{a^h}\nabla_h b$, we have : $$\|\Delta_q^v T_{a^h} \nabla_h b\|_{L_v^2}\leq \sum_{|q'-q|\leq 4}\|S_{q'-1}^v a^h\|_{L_v^\infty}\|\Delta_{q'}^v\nabla_h b\|_{L_v^2}.$$

    Hölder's inequality gives :
    $$\displaylines{\|\Delta_q^v T_{a^h} \nabla_h b\|_{L_h^{4/3}(L_v^2)}\leq \sum_{|q'-q|\leq 4}\|S_{q'-1}^v a^h\|_{L_h^4(L_v^\infty)} \|\Delta_{q'}^v\nabla_h b\|_{L^2(\Omega)}.}$$

   According to the corollary \ref{coro Bernstein-Gagliardo-Nirenberg}, we have : $$\displaylines{\|\Delta_q^v T_{a^h} \nabla_h b\|_{L_h^{4/3}(L_v^2)}\leq \|a^h\|_{\mathcal{B}^{0,1/2}}^{1/2}\|\nabla_h a^h\|_{\mathcal{B}^{0,1/2}}^{1/2} \sum_{|q-q'|\leq 4}\|\Delta_{q'}^v \nabla_h b\|_{L^2(\Omega)}.}$$ 

Using Hölder's inequality in time, we obtain :
$$\displaylines{\|\Delta_q^v T_{a^h}\nabla_h b\|_{L_T^{4/3}(L_h^{4/3}L_v^2)}\leq \|a^h\|_{L_T^\infty(\mathcal{B}^{0,1/2})}^{1/2}\|\nabla_h a^h\|_{L_T^2(\mathcal{B}^{0,1/2})}^{1/2} \hfill\cr\hfill \times \sum_{|q'-q|\leq 4}\|\Delta_{q'}^v\nabla_h b\|_{L_T^2(L^2)}.}$$
So we have: \begin{multline}\label{estimée paraproduit 1.1}\|\Delta_q^v T_{a^h}\nabla_h b\|_{L_T^{4/3}(L_h^{4/3}L_v^2)} \\ \leq 2^{-q/2}C c_{q}^{(2)}\|a^h\|_{L_T^\infty(\mathcal{B}^{0,1/2})}^{1/2}\|\nabla_h a^h\|_{L_T^2(\mathcal{B}^{0,1/2})}^{1/2}  \|\nabla_h b\|_{\tilde{L}_T^2(\mathcal{B}^{0,1/2})}.\end{multline}
\\
$\bullet$ For the second term $\Delta_q^v T_{\nabla_h b}a^h$, we have: $$\|\Delta_q^v T_{\nabla_h b}a^h\|_{L_h^{4/3}L_v^2}\leq \sum_{|q'-q|\leq 4}\|S_{q'-1}^v\nabla_h b\|_{L_h^2(L_v^\infty)}\|\Delta_{q'}^v a^h\|_{L_h^4(L_v^2)}.$$

Using Lemma \ref{ordre intégration}, \ref{Bernstein et Gagliardo-Nirenberg} and \ref{coro Bernstein-Gagliardo-Nirenberg} lead to : 
$$\|\Delta_q^v T_{\nabla_h b}a^h\|_{L_h^{4/3}L_v^2}\leq \|\nabla_h b\|_{\mathcal{B}^{0,1/2}}\sum_{|q'-q|\leq 4}\|\Delta_{q'}^v a^h\|_{L^2(\Omega)}^{1/2}\|\Delta_{q'}^v \nabla_h a^h\|_{L^2(\Omega)}^{1/2}.$$

By Hölder's inequality, we have :
$$\displaylines{\|\Delta_q^v T_{\nabla_h b}a^h\|_{L_T^{4/3}(L_h^{4/3}L_v^2)}\leq \|\nabla_h b\|_{L_T^2(\mathcal{B}^{0,1/2})}\sum_{|q'-q|\leq 4}\|\Delta_{q'}^v a^h\|_{L_T^\infty(L^2)}^{1/2}\hfill\cr\hfill\times\|\Delta_{q'}^v \nabla_h a^h\|_{L_T^2(L^2)}^{1/2}.}$$

By noting $\tilde{c}_q=\sqrt{\tilde{c}_q^{(1)}\tilde{c}_q^{(2)}}\in l_q^1$ where $\tilde{c}_q^{(1)}=\frac{2^{q/2}\|\Delta_q^v a^h\|_{L_T^\infty(L^2)}}{\|a^h\|_{\tilde{L}_T^\infty(\mathcal{B}^{0,1/2})}}$ and $\tilde{c}_q^{(2)}=\frac{2^{q/2}\|\Delta_q^v \nabla_h a^h\|_{L_T^2(L^2)}}{\|\nabla_h a^h\|_{\tilde{L}_T^2(\mathcal{B}^{0,1/2})}}$, we have : 

\begin{equation}\label{estimée paraproduit 1.2}\begin{aligned}\|\Delta_q^v T_{\nabla_h b}a^h\|_{L_T^{4/3}(L_h^{4/3}L_v^2)}\leq 2^{-q/2}C \tilde{c}_q \|\nabla_h b\|_{L_T^2(\mathcal{B}^{0,1/2})}\|a^h\|_{\tilde{L}_T^\infty(\mathcal{B}^{0,1/2})}^{1/2} \\ \times\|\nabla_h a^h\|_{\tilde{L}_T^2(\mathcal{B}^{0,1/2})}^{1/2}.\end{aligned}\end{equation}
\\
$\bullet$ For the remainder term $R(a^h,\nabla_h b)$, by Bernstein's inequality \eqref{classique GN}, we have :
$$\displaylines{\|\Delta_q^v R(a^h,\nabla_h b)\|_{L_v^2}\leq 2^{q/2}\|\Delta_q^v R(a^h,\nabla_h b)\|_{L_v^1}\leq 2^{q/2}\sum_{\underset{i\in\{0,\pm 1\}}{q'>q-4}}\|\Delta_{q'}^v a^h\|_{L_v^2}\hfill\cr\hfill \times\|\Delta_{q'-i}^v \nabla_h b\|_{L_v^2}.}$$

By the Hölder inequality and the Gagliardo-Nirenberg inequality \eqref{classique GN}, we obtain : 
$$\displaylines{\|\Delta_q^v R(a^h,\nabla_h b)\|_{L_h^{4/3}(L_v^2)}\leq 2^{q/2}\sum_{\underset{i\in\{0,\pm 1\}}{q'>q-4}}\|\Delta_{q'}^v a^h\|_{L_h^4(L_v^2)}\|\Delta_{q'-i}^v \nabla_h b\|_{L^2(\Omega)} \hfill\cr\hfill \leq 2^{q/2} \sum_{\underset{i\in\{0,\pm 1\}}{q'>q-4}}\|\Delta_{q'}^v a^h\|_{L^2(\Omega)}^{1/2} \|\Delta_{q'}^v \nabla_h a^h\|_{L^2(\Omega)}^{1/2}  \|\Delta_{q'-i}^v \nabla_h b\|_{L^2(\Omega)} .}$$

Taking the norm $L_T^{4/3}$, we have: 
$$\displaylines{\|\Delta_q^v R(a^h,\nabla_h b)\|_{L_T^{4/3}(L_h^{4/3}L_v^2)}\leq 2^{q/2}\sum_{\underset{i\in\{0,\pm 1\}}{q'>q-4}}\|\Delta_{q'}^v a^h\|_{L_T^\infty(L^2)}^{1/2}\|\Delta_{q'}^v \nabla_h a^h\|_{L_T^2(L^2)}^{1/2}\hfill\cr\hfill\times \|\Delta_{q'-i}^v \nabla_h b\|_{L_T^2(L^2)}.}$$

By noting $e_q=c_q^{(2)}\sqrt{\tilde{c}_q^{(1)}\tilde{c}_q^{(2)}}$, we have: $$\displaylines{\|\Delta_q^v R(a^h,\nabla_h b)\|_{L_T^{4/3}(L_h^{4/3}L_v^2)}\leq 2^{q/2}\sum_{\underset{i\in\{0,\pm 1\}}{q'>q-4}}2^{-q'}e_{q'}\|a^h\|_{\tilde{L}_T^\infty(\mathcal{B}^{0,1/2})}^{1/2} \hfill\cr\hfill\times\|\nabla_h a^h\|_{\tilde{L}_T^2(\mathcal{B}^{0,1/2})}^{1/2} \|\nabla_h b\|_{\tilde{L}_T^2(\mathcal{B}^{0,1/2})}.}$$

If we note $\tilde{e}_q=2^q \sum_{q'\geq q-N_0}2^{-q'}e_q$, then $\tilde{e}_q\in l_q^1$ and we have: \begin{multline}\label{terme de reste1}
   \|\Delta_q^v R(a^h, \nabla_h b)\|_{L_T^{4/3}(L_h^{4/3}L_v^2)} \\ \leq C2^{-q/2} \tilde{e}_q \|a^h\|_{\tilde{L}_T^\infty(\mathcal{B}^{0,1/2})}^{1/2}\|\nabla_h a^h\|_{\tilde{L}_T^2(\mathcal{B}^{0,1/2})}^{1/2} \|\nabla_h b\|_{\tilde{L}_T^2(\mathcal{B}^{0,1/2})}.
\end{multline}
Summing up the estimates \eqref{estimée paraproduit 1.1}, \eqref{estimée paraproduit 1.2} and \eqref{terme de reste1} and using Bony's decomposition, we obtain: 
$$\|F_q^h\|_{L_T^{4/3}(L_h^{4/3}L_v^2)}\leq C 2^{-q/2}d_q \|a^h\|_{\tilde{L}_T^\infty(\mathcal{B}^{0,1/2})}^{1/2}\|\nabla_h a^h\|_{\tilde{L}_T^2(\mathcal{B}^{0,1/2})}^{1/2} \|\nabla_h b\|_{\tilde{L}_T^2(\mathcal{B}^{0,1/2})}$$

We therefore have \begin{equation}\label{terme horizontal 1}\begin{aligned}
    \int_0^T |(F_q^h(\tau)|\Delta_q^v b(\tau))_{L^2}|d\tau \leq 2^{-q}C c_q   \|a^h\|_{\tilde{L}_T^\infty(\mathcal{B}^{0,1/2})}^{1/2}\|\nabla_h a^h\|_{\tilde{L}_T^2(\mathcal{B}^{0,1/2})}^{1/2} \\ \times \|b\|_{\tilde{L}_T^\infty(\mathcal{B}^{0,1/2})}^{1/2} \|\nabla_h b\|_{\tilde{L}_T^2(\mathcal{B}^{0,1/2})}^{3/2}.\end{aligned}
\end{equation}

\subsection*{Vertical term} It remains to estimate the term $\int_0^T|(\Delta_q^v(a^v\partial_z b)|\Delta_q^v b)|d\tau$.  

To do this, we will use the following decomposition given in \cite{Chemin}: 

\begin{equation}\label{décomposition Chemin}\Delta_q^v(a^v\partial_z b)\mathrel{\mathop:}=\sum_{i=1}^4 A_i,\end{equation}
where $$A_1\mathrel{\mathop:}= (S_{q-1}^v a^v)\partial_z \Delta_q^v b, \quad A_2\mathrel{\mathop:}= \sum_{|q'-q|\leq 4}[\Delta_q^v;S_{q'-1}^v a^v]\partial_z \Delta_{q'}^v b,$$
$$A_3\mathrel{\mathop:}= \sum_{|q'-q|\leq 4}((S_{q-1}^v a^v-S_{q'-1}^v a^v)\partial_z \Delta_q^v \Delta_{q'}^v b)$$ and $$A_4\mathrel{\mathop:}= \sum_{q'> q+4}\Delta_q^v(S_{q'+1}^v(\partial_z b)\Delta_{q'}^v a^v).$$
\\
The general idea is to transform by integration by parts $a^v \partial_z b$ into $(\partial_z a^v) \times b $ and to use the fact that $\dive a=0$ to get $\dive_h a^h$ and end up with estimates close to those studied previously. To be able to do this with the localisation operators, we use the previous decomposition. We will use integration by parts for the first term. For the second term, we will use commutator estimates to obtain $\partial_z a^v$. By frequency localisation, the third term will be treated in a similar way to the second. For the last term, the aim is to be able to "move" the $\partial_z$ derivative on $a^v$ using Bernstein's lemma and frequency localisations.
\\
$\bullet$ Let us start by estimating the first term $\int_0^T|(A_1|\Delta_q^v b)_{L^2}|d\tau.$

First, we have integration by parts and the divergence-free condition: $$((S_{q-1}^v a^v) \partial_z \Delta_q^v b| \Delta_q^v b)_{L^2}=\frac{1}{2}(S_{q-1}^v(\dive_h a^h)\Delta_q^v b|\Delta_q^v b)_{L^2}.$$

By Hölder's inequality and the corollary \ref{coro Bernstein-Gagliardo-Nirenberg}, we then obtain:
\begin{align*}
    \int_0^T|(A_1|\Delta_q^v b)_{L^2}|d\tau & \leq C \int_0^T \|S_{q-1}^v (\nabla_h a^h)\|_{L_v^\infty L_h^2}\|\Delta_q^v b\|_{L_v^2 L_h^4}^2 d\tau 
    \\ & \leq C \int_0^T \|\nabla_h a^h(\tau)\|_{\mathcal{B}^{0,1/2}}\|\Delta_q^v b(\tau)\|_{L_v^2 L_h^4}^2 d\tau.
\end{align*}

By the Gagliardo-Nirenberg inequality \eqref{classique GN} and the Hölder inequality in time, we obtain
\begin{multline}\label{premier lemme de convection premier terme}\int_0^T|(A_1|\Delta_q^v b)_{L^2}|d\tau \\ \leq C 2^{-q} c_q \|\nabla_h a^h\|_{\tilde{L}_T^2(\mathcal{B}^{0,1/2})}\|b\|_{\tilde{L}_T^\infty(\mathcal{B}^{0,1/2})} \|\nabla_h b\|_{\tilde{L}_T^2(\mathcal{B}^{0,1/2})}.\end{multline}
\\
$\bullet$ Let us estimate the second term
$\int_0^T|(A_2|\Delta_q^v b)_{L^2}|d\tau$.

By Hölder's inequality and the commutator estimates of Lemma \ref{estimée de commutateur}, we have: 
\begin{align*}
     & \int_0^T|(A_2|\Delta_q^v b)_{L^2}|d\tau
    \\ \leq & \sum_{|q'-q|\leq 4}\int_0^T \|[\Delta_q^v;S_{q'-1}^v a^v]\partial_z \Delta_{q'}^v b\|_{L_v^2(L_h^{4/3})}\|\Delta_q^v b\|_{L_v^2 L_h^4}d\tau
    \\ \leq & C \sum_{|q'-q|\leq 4} 2^{-q} \int_0^T \|S_{q'-1}\partial_z a^v\|_{L_v^\infty L_h^2} \|\partial_z \Delta_{q'}^v b\|_{L_v^2 L_h^4}\|\Delta_q^v b\|_{L_v^2 L_h^4}d\tau
    \\ \leq & C \sum_{|q'-q|\leq 4}2^{-q}\|S_{q'-1}^v \partial_z a^v\|_{L_T^2(L_v^\infty L_h^2)}\|\partial_z \Delta_{q'}^v b\|_{L_T^4(L_v^2 L_h^4)}\|\Delta_{q'}^v b\|_{L_T^4(L_v^2 L_h^4)}.
\end{align*}

By Bernstein and Gagliardo-Nirenberg inequalities \eqref{classique GN}, we have: \begin{align*}\|\partial_z \Delta_{q'}^v b\|_{L_T^4(L_v^2 L_h^4)} & \leq C 2^{q'}\|\Delta_{q'}^v b\|_{L_T^4(L_v^2 L_h^4)} \\ & \leq C2^{q'}\|\Delta_{q'}^v b\|_{L_T^\infty(L^2)}^{1/2} \|\Delta_{q'}^v \nabla_h b\|_{L_T^2(L^2)}^{1/2}.\end{align*}

As $\partial_z a^v=-\dive_h a^h$, we obtain that: \begin{multline}\label{premier lemme de convection deuxième terme}\int_0^T|(A_2|\Delta_q^v b)_{L^2}|d\tau \\ \leq C 2^{-q} c_q \|\nabla_h a^h\|_{\tilde{L}_T^2(\mathcal{B}^{0,1/2})} \|b\|_{\tilde{L}_T^\infty(\mathcal{B}^{0,1/2})}\|\nabla_h b\|_{\tilde{L}_T^2(\mathcal{B}^{0,1/2})}.\end{multline}
\\
$\bullet$ For the third term $\int_0^T|(A_3|\Delta_q^v b)_{L^2}|d\tau,$ 
we note that $$\Delta_q^v \Delta_{q'}^v=0 \quad \text{for} \ |q'-q|\geq 2$$ and $S_{q-1}^v a^v-S_{q'-1}^v a^v$ is either zero or $a^v$ is located on an annulus of size $2^q$ for $|q'-q|\leq 1$.
On the other hand, using $\dive a=0$ and the same calculations as for the second term, we can easily obtain the estimate :
\begin{multline}\label{premier lemme de convection troisième terme}\int_0^T|(A_3|\Delta_q^v b)_{L^2}|d\tau \\ \leq C 2^{-q} c_q \|\nabla_h a^h\|_{\tilde{L}_T^2(\mathcal{B}^{0,1/2})} \|b\|_{\tilde{L}_T^\infty(\mathcal{B}^{0,1/2})}\|\nabla_h b\|_{\tilde{L}_T^2(\mathcal{B}^{0,1/2})}.\end{multline}
\\
$\bullet$ The last term to be estimated is $\int_0^T|(A_4|\Delta_q^v b)_{L^2}|d\tau$. By Hölder inequality, we have:
$$\displaylines{\sum_{q'>q-4}|(\Delta_q^v(S_{q'+1}^v(\partial_z b)\Delta_{q'}^v a^v)|\Delta_{q}^v b)_{L^2}|\leq \sum_{q'>q-4}\|S_{q'+1}^v(\partial_z b)\|_{L_T^4(L_v^\infty L_h^4)} \hfill\cr\hfill\times\|\Delta_{q'}^v a^v\|_{L_T^2(L_v^2 L_h^2)}\|\Delta_q^v b\|_{L_T^4(L_v^2 L_h^4)}.}$$

By Bernstein's Lemma \ref{Bernstein et Gagliardo-Nirenberg} and Corollary \ref{coro Bernstein-Gagliardo-Nirenberg}, we have: 
$$\displaylines{\|S_{q'+1}^v(\partial_z b)\|_{L_T^4(L_v^\infty L_h^4)}\leq C 2^{q'}\|b\|_{L_T^4(L_v^\infty L_h^4)}\leq  C 2^{q'} \|b\|_{L_T^\infty(\mathcal{B}^{0,1/2})}^{1/2} \hfill\cr\hfill\times \|\nabla_h b\|_{L_T^2(\mathcal{B}^{0,1/2})}^{1/2}.}$$

By Bernstein's lemma \ref{Bernstein et Gagliardo-Nirenberg} and the divergence-free condition, we have : 
\begin{align*}\|\Delta_{q'}^v a^v\|_{L^2([0,T]\times \Omega)} & \leq C 2^{-q'}\|\Delta_{q'}^v \partial_z a^v\|_{L^2([0,T]\times \Omega)} \\ & \leq C 2^{-q'}\|\Delta_{q'}^v \nabla_h a^h\|_{L^2([0,T]\times \Omega)}.\end{align*}

This means that $$\displaylines{\int_0^T|(A_1|\Delta_q^v b)_{L^2}|d\tau \hfill\cr \leq C\|b\|_{L_T^\infty(\mathcal{B}^{0,1/2})}^{1/2}\|\nabla_h b\|_{L_T^2(\mathcal{B}^{0,1/2})}^{1/2} \sum_{q'>q-4} \|\Delta_{q'}^v \nabla_h a^h\|_{L_T^2(L^2(\Omega))}\|\Delta_q^v b\|_{L_T^\infty(L^2)}^{1/2} \hfill\cr\hfill \times\|\Delta_q^v \nabla_h b\|_{L_T^2(L^2)}^{1/2} \cr \leq C 2^{-q/2}c_q \|b\|_{L_T^\infty(\mathcal{B}^{0,1/2})}\|\nabla_h b\|_{L_T^2(\mathcal{B}^{0,1/2})}\sum_{q'\geq q-4}2^{-q'/2}\tilde{c}_{q'}^{(2)}\|\nabla_h a^h\|_{\tilde{L}_T^2(\mathcal{B}^{0,1/2})}.}$$ 
We then have:
\begin{multline}\label{premier lemme de convection quatrième terme}
\int_0^T|(A_4|\Delta_q^v b)_{L^2}|d\tau \\ \leq C2^{-q}d_q \overline{b}_q \|b\|_{L_T^\infty(\mathcal{B}^{0,1/2})}\|\nabla_h b\|_{L_T^2(\mathcal{B}^{0,1/2})}\|\nabla_h a^h\|_{\tilde{L}_T^2(\mathcal{B}^{0,1/2})},    
\end{multline}
where $\overline{b}_q=2^{q/2}\sum_{q'>q-4}2^{-q'/2}\tilde{c}_{q'}^{(2)}$.

By estimates \eqref{premier lemme de convection premier terme}, \eqref{premier lemme de convection deuxième terme}, \eqref{premier lemme de convection troisième terme} and \eqref{premier lemme de convection quatrième terme}, we deduce: 
\begin{equation}\label{terme vertical}\begin{aligned}\int_0^T|(\Delta_q^v(a^v\partial_z b)|\Delta_q^v b)|d\tau\leq C 2^{-q} c_q  C \|\nabla_h a^h\|_{\tilde{L}_T^2(\mathcal{B}^{0,1/2})} \\ \times\|\nabla_h b\|_{\tilde{L}_T^2(\mathcal{B}^{0,1/2})}\|v\|_{L_T^\infty(\mathcal{B}^{0,1/2})}.\end{aligned}\end{equation}

    So, by \eqref{terme horizontal 1} and \eqref{terme vertical}, we deduce the lemma.
\end{proof}

We therefore reprove the corresponding lemma for the domain \( \Omega_2 \), which highlights the main difficulty: estimating the convection term under the assumption that horizontal mean values vanish. These mean values can be managed manually as in \cite{Paicu2}.

\begin{lemma}\label{lemme convection final}
Let $\Omega=\Omega_1$ be the three-dimensional torus.
    There exists a constant $C$ such that for any time-dependent vector fields, $a$ and $b$ with $a^h$ and $b$ in $E$ (defined in \eqref{Espace fonctionnel E}) with $a$ of zero divergence, there exists a sequence $c_q$ verifying $\sum_q \sqrt{c_q}\leq 1$ and $$\displaylines{\int_0^T \left|\Delta_q^v(a\cdot\nabla b)|\Delta_q^v b)_{L^2(\Omega)}\right|d\tau \hfill\cr\hfill \leq C c_q 2^{-q}\bigg(\|\nabla_h a^h\|_{\tilde{L}_T^2(\mathcal{B}^{0,1/2})}\|b\|_{\tilde{L}_T^\infty(\mathcal{B}^{0,1/2})}\|\nabla_h b\|_{\tilde{L}_T^2(\mathcal{B}^{0,1/2})} \hfill\cr\hfill+\|a^h\|_{\tilde{L}_T^\infty(\mathcal{B}^{0,1/2})}\|\nabla_h b\|_{\tilde{L}_T^2(\mathcal{B}^{0,1/2})}^2\bigg).}$$ 
\end{lemma}

The proof is very similar to that of Lemma 6 of \cite{Paicu2}. The details can be found in \cite{Thèse}.

In the case of the decomposition $u^h=\tilde{u}^h+\overline{u}^h$ where $\tilde{u}^h$ and $\overline{u}^h$ satisfy \eqref{Equations primitives sans moyenne verticale} and \eqref{Navier-Stokes 2D} respectively, we also needed the following convection lemma:
\begin{lemma}\label{3eme terme de convection}
For $\Omega=\Omega_1$ or $\Omega_2$, we have the following inequality for $T>0$:
    \begin{equation}\label{terme de convection2}\begin{aligned}\int_0^T |(\Delta_q^v \tilde{u}^h\cdot\nabla_h\overline{u}^h|\Delta_q^v \tilde{u}^h)_{L^2}|dt\leq 2^{-q}c_q \|\nabla_h \overline{u}^h\|_{L_T^2(L_h^2)}\|\tilde{u}^h\|_{\tilde{L}_T^\infty(\mathcal{B}^{0,1/2})} \\ \times \|\nabla_h \tilde{u}^h\|_{\tilde{L}_T^2(\mathcal{B}^{0,1/2})},\end{aligned}\end{equation}  with $\sum_q \sqrt{c_q}\leq 1$.
\end{lemma}

\begin{proof}
    Since $\overline{u}^h$ does not depend on the vertical variable, we have: $$|(\Delta_q^v(\tilde{u}^h\cdot\nabla_h\overline{u}^h)|\Delta_q^v \tilde{u}^h)_{L^2}|=|(\Delta_q^v\tilde{u}^h)\cdot\nabla_h\overline{u}^h)|\Delta_q^v \tilde{u}^h)_{L^2}|.$$

    We need to distinguish between the $\Omega_1$ and $\Omega_2$ cases. Let us start with the easiest case $\Omega_2$ where the horizontal mean value does not come into play:

    \begin{enumerate}
        \item For the case $\Omega_2$, we have by Hölder's inequality: $$|(\Delta_q^v \tilde{u}^h\cdot\nabla_h \overline{u}^h)|\Delta_q^v \tilde{u}^h)_{L^2}|\leq \|\nabla_h \overline{u}^h\|_{L_h^2}\|\Delta_q^v \tilde{u}^h\|_{L_v^2 L_h^4}^2.$$
    Using the Gagliardo-Nirenberg inequality, we have by \eqref{classique GN} : 
 $$|(\Delta_q^v \tilde{u}^h\cdot \nabla_h \overline{u}^h|\Delta_q^v \tilde{u}^h)_{L^2}| \leq C \|\Delta_q^v \tilde{u}^h\|_{L^2} \|\nabla_h \Delta_q^v \tilde{u}^h\|_{L^2}\|\nabla_h \overline{u}^h\|_{L^2(\Omega_h)}.$$ 

    Using Hölder's inequality in time, we obtain: 
    $$\displaylines{\int_0^T|(\Delta_q^v \tilde{u}^h\cdot\nabla_h \overline{u}^h|\Delta_q^v \tilde{u}^h)_{L^2}|dt\leq 2^{-q}c_q \|\nabla_h \overline{u}^h\|_{L_T^2(L_h^2)}\|\tilde{u}^h\|_{\tilde{L}_T^\infty(\mathcal{B}^{0,1/2})} \hfill\cr\hfill \times\|\nabla_h \tilde{u}^h\|_{\tilde{L}_T^2(\mathcal{B}^{0,1/2})}.}$$
        \item Let us now consider the $\Omega_1$ case: let $w\mathrel{\mathop:}=\tilde{u}^h$ and decompose it with its horizontal mean value.
Let us define $$\underline{w}(x_v)\mathrel{\mathop:}=\frac{1}{4}\int w dx_h\quad \text{and} \quad \tilde{w}\mathrel{\mathop:}=w-\underline{w}.$$

The fields $\Delta_q^v \underline{w}$ and $\Delta_q^v \tilde{w}$ are orthogonal in $L^2$ and we have : $$\sum_q 2^{q/2}\|\Delta_q^v \underline{w}\|_{L_v^2}\leq C \|w\|_{\mathcal{B}^{0,1/2}} \quad  \text{respectively} \quad \|\tilde{w}\|_{\mathcal{B}^{0,1/2}}\leq C \|w\|_{\mathcal{B}^{0,1/2}}.$$

We then have the following decomposition: 
$$\displaylines{(\Delta_q^v \tilde{u}^h\cdot\nabla_h \overline{u}^h|\Delta_q^v \tilde{u}^h)_{L^2}=(\Delta_q^v \underline{w}^h\cdot\nabla_h \overline{u}^h|\Delta_q^v \underline{w})_{L^2}+(\Delta_q^v \tilde{w}\cdot \nabla_h \overline{u}^h| \Delta_q^v \tilde{w})_{L^2}\hfill\cr\hfill+2(\Delta_q^v \underline{w}\cdot\nabla_h \overline{u}^h|\Delta_q^v \tilde{w})_{L^2}.}$$
\\
$\bullet$ By integration by parts ($\underline{w}^h$ depends only on the vertical variable), we have:
$$(\Delta_q^v \underline{w}^h\cdot\nabla_h \overline{u}^h|\Delta_q^v \underline{w})_{L^2}=0.$$
\\
$\bullet$ Using the calculations for $\Omega_2$ (since $\tilde{w}$ has zero horizontal mean value) : $$\displaylines{\int_0^T|((\Delta_q^v \tilde{w}\cdot\nabla_h) \overline{u}^h|\Delta_q^v \tilde{w})_{L^2}|dt\leq 2^{-q}c_q \|\nabla_h \overline{u}^h\|_{L_T^2(L_h^2)}\|\tilde{u}^h\|_{\tilde{L}_T^\infty(\mathcal{B}^{0,1/2})} \hfill\cr\hfill\times\|\nabla_h \tilde{u}^h\|_{\tilde{L}_T^2(\mathcal{B}^{0,1/2})}.}$$
    \\
$\bullet$ For the last term, applying Hölder's inequality, we obtain: $$\int \Delta_q^v \underline{w}\cdot\nabla_h \overline{u}^h\Delta_q^v \tilde{w}dx\leq \|\Delta_q^v \underline{w}\|_{L_v^2}\|\nabla_h \overline{u}^h\|_{L_h^2}\|\Delta_q^v \tilde{w}\|_{L^2}.$$

    We then have by Hölder's inequality in time:
    $$\displaylines{\int_0^T ((\Delta_q^v \underline{w}\cdot\nabla_h) \overline{u}^h|\Delta_q^v \tilde{w})_{L^2}|dt\leq 2^{-q}c_q \|\nabla_h \overline{u}^h\|_{L_T^2(L_h^2)}\|\tilde{u}^h\|_{\tilde{L}_T^\infty(\mathcal{B}^{0,1/2})} \hfill\cr\hfill\times\|\nabla_h \tilde{u}^h\|_{\tilde{L}_T^2(\mathcal{B}^{0,1/2})}.}$$
\end{enumerate}
    Finally, we obtain \eqref{terme de convection2}.
\end{proof}

\end{document}